\documentclass[10pt]{amsart}

\makeatletter
\def\@settitle{\begin{center}%
  \baselineskip14\p@\relax
    \normalfont\LARGE%<- NEW

  \@title
  \end{center}%
}
\makeatother
\usepackage{epigraph}
\usepackage[dvipsnames]{xcolor}
\usepackage[hyperfootnotes=false]{hyperref}
\usepackage{enumerate}
\usepackage{tikz}
\usepackage{tikz-cd}
\usepackage{hyperref}
\usepackage{amsthm,amsfonts,amsmath,amscd,amssymb}
\usepackage{xcolor,float}
\usepackage[all]{xy}
\usepackage{mathrsfs}
\usepackage{graphics,graphicx}
\usepackage{caption}
\usepackage{comment}
\usepackage{array}
\newcolumntype{P}[1]{>{\centering\arraybackslash}p{#1}}
\newcolumntype{M}[1]{>{\centering\arraybackslash}m{#1}}

\usepackage{epigraph}
\let\oldmarginpar\marginpar
\renewcommand\marginpar[1]{\-\oldmarginpar[\raggedleft\footnotesize #1]%
	{\raggedright\footnotesize #1}}
\usepackage{cleveref}
\usepackage{thmtools}
\theoremstyle{plain}

\newtheorem{thm}{Theorem}[section]
\newtheorem{lemma}[thm]{Lemma}
\newtheorem*{theorem*}{Theorem}
\newtheorem*{corollary*}{Corollary}

\newtheorem{prop}[thm]{Proposition}

\theoremstyle{definition}

\newtheorem{remark}[thm]{Remark}

\numberwithin{equation}{section}

\newcommand{\Z}{\mathbb{Z}}
\newcommand{\Q}{\mathbb{Q}}
\newcommand{\R}{\mathbb{R}}
\newcommand{\C}{\mathbb{C}}

\newcommand{\La}{\Lambda}

\newcommand{\la}{\lambda}

\newcommand{\sse}{\subset}
\newcommand{\lr}{\longrightarrow}

\newcommand{\sm}{\setminus}

\newcommand{\wt}{\widetilde}

\newcounter{daggerfootnote}

\usepackage{dsfont}
\allowdisplaybreaks

\def \vertbar [#1](#2,#3,#4){
    \draw [#1] (#2,#3) -- (#2,#4);
    \draw [fill=white] (#2,#3) circle [radius=0.1];
    \draw [fill=black] (#2,#4) circle [radius=0.1];
}

\usepackage{mathdots}
\usepackage{afterpage}
\makeatletter
\providecommand{\leftsquigarrow}{%
  \mathrel{\mathpalette\reflect@squig\relax}%
}
\newcommand{\reflect@squig}[2]{%
  \reflectbox{$\m@th#1\rightsquigarrow$}%
}
\makeatother

\makeatletter
\def\Ddots{\mathinner{\mkern1mu\raise\p@
\vbox{\kern7\p@\hbox{.}}\mkern2mu
\raise4\p@\hbox{.}\mkern2mu\raise7\p@\hbox{.}\mkern1mu}}
\makeatother

\def \horline [#1](#2,#3,#4){
    \draw [#1] (#2,#4) -- (#3,#4);
    \draw [fill=white] (#2,#4) circle [radius=0.1];
    \draw [fill=black] (#3,#4) circle [radius=0.1];
}

\def \crossing (#1,#2)(#3,#4){
\draw (#1,#2) -- (#3,#4);
\draw (#1,#4) -- (#3,#2);
}

\DeclareFontFamily{U}{mathb}{}
\DeclareFontShape{U}{mathb}{m}{n}{
  <-5.5> mathb5
  <5.5-6.5> mathb6
  <6.5-7.5> mathb7
  <7.5-8.5> mathb8
  <8.5-9.5> mathb9
  <9.5-11.5> mathb10
  <11.5-> mathbb12
}{}

\usetikzlibrary{decorations.markings, arrows}
\tikzset{tangent/.style={decoration={markings,mark=at position #1 with {
      \coordinate (tangent point-\pgfkeysvalueof{/pgf/decoration/mark info/sequence number}) at (0pt,0pt);
      \coordinate (tangent unit vector-\pgfkeysvalueof{/pgf/decoration/mark info/sequence number}) at (1,0pt);
      \coordinate (tangent orthogonal unit vector-\pgfkeysvalueof{/pgf/decoration/mark info/sequence number}) at (0pt,1);
      }},postaction=decorate},
    use tangent/.style={
        shift=(tangent point-#1),
        x=(tangent unit vector-#1),
        y=(tangent orthogonal unit vector-#1)
    },
    use tangent/.default=1
    }


\begin{document}

\title{The Legendrian Hopf Link has exactly two Lagrangian fillings}
	
	%\subjclass[2010]{Primary: 53D12. Secondary: 57K33, 13F60.}

    \author{Bryce Thomson}
	\address{University of California Davis, Dept. of Mathematics, USA}
	\email{bthomson@ucdavis.edu}

\begin{abstract}
We prove that there are precisely two embedded exact Lagrangian fillings of the standard Legendrian Hopf link, up to compactly supported Hamiltonian isotopy. It was known that the standard Legendrian Hopf link admitted at least two such Lagrangian fillings: we show these are all. Specifically, we use a type of neck-stretching procedure to construct a pseudoholomorphic conic fibration that makes a given arbitrary exact Lagrangian filling fiber over a real curve, under a global pseudoholomorphic Lefschetz fibration. This then allows for an explicit Hamiltonian isotopy to be constructed from any given Lagrangian filling to one of two known standard fillings.
\end{abstract}

\maketitle

%\vspace{1cm}

%%%%%%%%%%%%%%%%%%%%%%%%%%%%%%%%%%%%%%%%%%%%%%%%%%%%%%%%%%%%%%%%%
%%%%%%%%%%%%%%%%%%%%%%%%%%%%%%%%%%%%%%%%%%%%%%%%%%%%%%%%%%%%%%%%%
%%%%%%%%%%%%%%%%%%%%%%%%%%%%%%%%%%%%%%%%%%%%%%%%%%%%%%%%%%%%%%%%%

%%%%%%%%%%%%%%%%%%%%%%%%%%%%%%%%%%%%%%%%%%%%%%%%%%%%%%%%%%%%%%%%%
%%%%%%%%%%%%%%%%%%%%%%%%%%%%%%%%%%%%%%%%%%%%%%%%%%%%%%%%%%%%%%%%%
%%%%%%%%%%%%%%%%%%%%%%%%%%%%%%%%%%%%%%%%%%%%%%%%%%%%%%%%%%%%%%%%%
%\vspace{-1cm}
\section{Introduction}\label{sec:intro}
In this paper we complete the classification of embedded exact Lagrangian fillings of the Legendrian Hopf link, up to compactly supported Hamiltonian isotopy in the standard symplectic Darboux ball. A front for this Legendrian link in the standard contact Darboux ball is depicted in Figure \ref{fig:maxtb_hopf}. The first step in the argument is the construction of a pseudoholomorphic Lefschetz fibration on the symplectic Darboux ball which is compatible with a given Lagrangian filling, so that it fibers over a real curve in the real plane under such a fibration. The second step shows that there are only two Hamiltonian isotopy classes of Lagrangian fillings that fiber over a real curve with the given boundary conditions.

\subsection{Scientific Context}\label{sec:scientificContext}
The study of exact Lagrangian fillings of Legendrian links has a prominent role in contact and symplectic topology, cf.~e.g.~\cite{MR3402346,MR4358415,Casals24_ClusterSeed,Casals22_braidloopsinfinitemonodromy,Chekanov97_LDGA,Ekholm12_LagrangianCobord,EGH00_SFT,MR4194293}. They can be used to study the Legendrian contact dg-algebra and its augmentations, see e.g.~\cite{Chekanov97_LDGA,MR4575870}, normal rulings of Legendrian fronts \cite{MR3335247,MR3784016} and Weinstein 4-manifolds and their Lagrangian skeleta, see e.g.~\cite{Casals19_LegFronts,Casals20_LagrangianSkeleta}. Recently, Lagrangian fillings have also been used to establish results on cluster algebras, see e.g.~\cite{casals2023demazureweavesreducedplabic,Casals24_Microlocal,Casals24_ClusterBraidVarieties}.

Even if the classification of exact Lagrangian fillings lies at the heart of low-dimensional contact and symplectic topology, it is still largely unresolved. To wit, the only Legendrian link for which we have a complete non-empty classification of Lagrangian fillings is the max-tb Legendrian unknot. This was a groundbreaking result proven in 1996, cf.~\cite[Theorem 1.1.A]{Eliashberg96_LocalLagr2Knots}. Many of the results since have focused on constructing and distinguishing Lagrangian fillings: the former via geometric methods, the latter via Floer-theoretic and sheaf-theoretic invariants.

Recently, R.~Casals introduced a precise conjecture on the classification of Lagrangian fillings for Legendrian links, based on his work on Lagrangian fillings and cluster algebras, see e.g.\cite{MR4358415,Casals24_ClusterSeed,Casals24_ClusterBraidVarieties} and the recent lecture series \cite{casals_cbms}. The scope of the conjecture is large but, in particular, it conjectures that the standard Legendrian Hopf Link should have exactly two Lagrangian fillings, as the cluster algebra of mutable type $A_1$ has exactly two cluster seeds. Two such Lagrangian fillings were constructed and distinguished in \cite{Ekholm12_LagrangianCobord,Pan17_ExactLagrFillings}, and thus it remains to prove that any other embedded exact Lagrangian filling ought to be compactly supported Hamiltonian isotopic to one of these two. The goal of this article is to establish such fact, thus completing the classification of Lagrangian fillings of the standard Legendrian Hopf link.

In order to establish such a classification, we construct a pseudoholomorphic Lefschetz fibration on the standard Darboux ball so that a given Lagrangian filling fibers over a real embedded curve in the real plane under such a fibration. This is inspired by the technique of classification for Lagrangian tori near a Lagrangian Whitney sphere in \cite{DR19_WhitneySphere}, see \Cref{fig:whitney_sphere}. In ibid., the objects studied are (not necessarily exact, but closed) Lagrangian tori in the complement of a regular conic. A pseudoholomorphic Lefschetz fibration is constructed so that the Lagrangian tori fiber over a circle in the real plane. We rely on some of the pseudoholomorphic analytic facts established in that paper, cf.~ \Cref{conicpencil}. Our classification result can be considered as improving the results in the aforementioned paper to the case of Lagrangians with Legendrian boundary. These differences are mostly in the geometric side, though we must also analyze the behavior of pseudoholomorphic curves when neck-stretching in this relative context. In particular, our Hamiltonian isotopies ought to be compactly supported, i.e.~relative to the boundary.

\begin{center}
	\begin{figure}[h!]
		\centering
		\includegraphics[scale=.8]{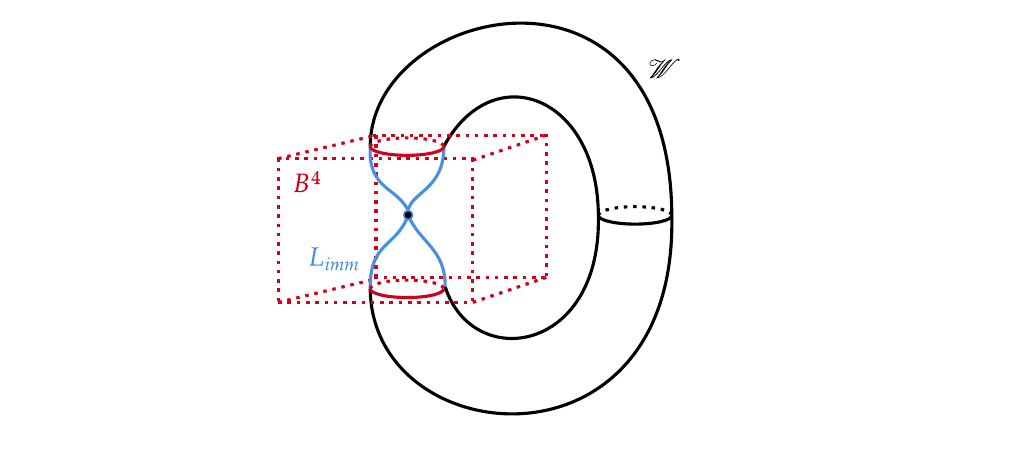}
		\caption{A depiction of the Whitney sphere $\mathscr{W}$ considered in \cite{DR19_WhitneySphere}. The Lagrangian fillings of the Hopf link can be understood to be inside a Darboux ball centered at the singularity, with that Lagrangian skeleton providing an immersed Lagrangian filling $L_{imm}$.}
		\label{fig:whitney_sphere}
	\end{figure}
\end{center}

%There are differences, however, between our setting and that of \cite{DR19_WhitneySphere}, which lead to some geometric discoveries and some different approaches which we discuss briefly in the next two paragraphs. 

%One insight from \cite{DR19_WhitneySphere} is that by removing a standard conic from the Darboux ball, the Lagrangian tori considered fiber over an embedded circle instead of an immersed one. A key geometric discovery in our paper is that the assumption of open, exact Lagrangians allows for the same result of fibering over an embedded curve, while our Lagrangians are contained in the honest Darboux ball. 

%Further, following the construction of the appropriate fibration, Dimitroglou Rizell uses a more global approach to construct the Hamiltonian isotopies needed using the Nearby Lagrangian Conjecture for the cotangent bundle of the torus proven in the same paper \cite{DR19_WhitneySphere}. The assumptions on our Lagrangian fillings allow for a more local, explicit approach to the construction of our Hamiltonian isotopies, cf. \Cref{sec:hamiso}.

\subsection{Main Result}
An oriented embedded exact Lagrangian filling $L$ of a Legendrian link $\Lambda\sse (S^3,\xi_{st})$ is an oriented embedded exact Lagrangian submanifold of the symplectization of $(S^3,\xi_{st})$ which is asymptotic to $\Lambda$ in a collar neighborhood of $S^3$, cf.~ \Cref{sec:Lagr_def} for a precise definition. The focus of this article is on the Legendrian link $\La_h$ whose front is in \cref{fig:maxtb_hopf}: the standard Legendrian Hopf link, with two max-tb unknot components. We consider the Legendrian link $\La_h\sse(S^3,\xi_{st})$ in the standard contact 3-sphere and its Lagrangian fillings in the standard symplectic Darboux ball $(B^4,\omega_{0})$, understood as an exact symplectic filling of $(S^3,\xi_{st})$. Given that we already have at least two distinct embedded fillings of $\La_h$ up to Hamiltonian isotopy, see e.g.~\Cref{sec:sheafClassification}, our goal is to argue that any other Lagrangian filling falls into one of these two Hamiltonian isotopy classes. This is the content of our main result:

\begin{thm}\label{thm:main}
Let $\La_h\sse(S^3,\xi_{st})$ be the Legendrian Hopf link, whose front is depicted in \Cref{fig:maxtb_hopf}. Then, there exist exactly two oriented embedded exact Lagrangian fillings of $\La_h$ in the standard Darboux ball $(B^4,\omega_{0})$, up to compactly supported Hamiltonian isotopy.
\end{thm}

\begin{center}
	\begin{figure}[h!]
		\centering
		\includegraphics[scale=.8]{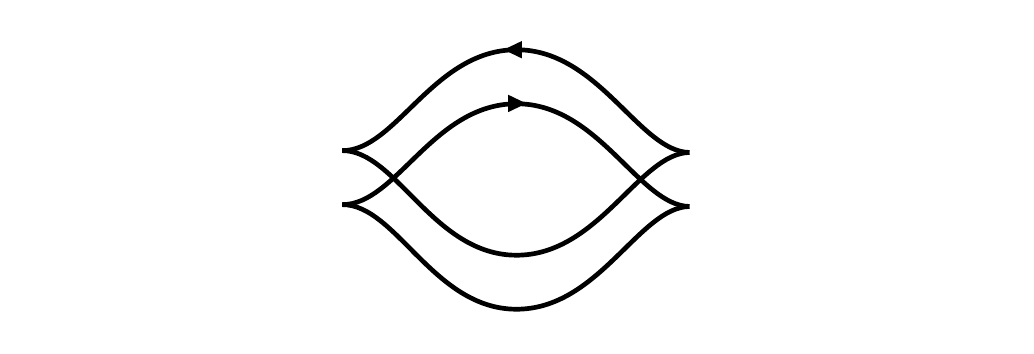}
		\caption{The front projection of the standard Legendrian Hopf link $\Lambda_h\sse (S^3,\xi_{st})$.\vspace{.75cm}}
		\label{fig:maxtb_hopf}
	\end{figure}
\end{center}

{\bf Acknowledgments}. I would like to first and foremost acknowledge the many invaluable conversations with my PhD advisor, Roger Casals, which have led to this paper. I would like to thank him for introducing me to the world of Lagrangian fillings, and for so generously sharing with me his time, expertise, and suggestions.
I would also like to thank Alex Simons for the many interesting and useful discussions regarding this project as well as Legendrians and Lagrangian fillings as a whole.
I am eternally grateful to my family: Mom, Dad, Paisley, Cameron, Grammy, and Poppa for their unconditional love and support and their never-ending encouragement. I would not be where I am today without them.

{\bf Structure of the manuscript}. The content of each section is as follows:
\begin{itemize}
    \item[(i)] \Cref{sec:setting} starts with the basic definitions required from symplectic and contact topology. It then discusses the standard Lefschetz fibration on $(\C^2,\omega)$ and how this gives us a first geometric picture: the max-$tb$ Legendrian Hopf link can be seen as a fibered link with respect to this fibration restricted to $(S^3, \xi_{st})\sse (\C^2,\omega)$ and two standard Lagrangian fillings can be constructed by choosing them to be smooth $S^1$-bundles over open curves in the base of this fibration. The section ends with a discussion of the necessary ingredients on pseudoholomorphic curves and \Cref{thm:DRfibr}, which establishes that pseudoholomorphic conic fibrations with properties similar to the standard Lefschetz fibration persist for any choice of tame almost complex structure.\newline

    \item[(ii)] \Cref{sec:fibration} executes one of the main ideas of the paper: given any embedded exact Lagrangian filling $L$ of the max-$tb$ Legendrian Hopf link, we can find a pseudoholomorphic conic fibration such that $L$ is graphical with respect to this fibration, in that it fibers over a real curve in the plane. The proof relies on a neck-stretching procedure, which, broadly, allows us to obtain broken pseudoholomorphic conics which break along closed curves (geodesics) on the filling $L$. The broken pieces can then be put back together to construct the desired global fibration.\newline

    \item[(iii)] \Cref{sec:hamiso} uses the compatible fibration established in \Cref{sec:fibration} in order to construct explicit Hamiltonian isotopies that bring any filling $L$ to one of the two standard fillings. This concludes the proof of \Cref{thm:main}.
\end{itemize}

\section{Preliminary ingredients}\label{sec:setting}

This section introduces the necessary constructions and results needed to establish \Cref{thm:main}. They are used in \Cref{sec:fibration} and \Cref{sec:hamiso} to prove the theorem.

\subsection{Standard Lefschetz fibrations}\label{sec:std_lef_fibr} Consider the exact symplectic manifold $(\C^2,\omega)$ with $\omega= \frac{i}{2}(dz_1\wedge d\overline{z_1}+dz_2\wedge d\overline{z_2})$ and its standard  Liouville form $\lambda=z_1d\overline{z_1}+z_2d\overline{z_2}$. The standard Lefschetz fibration of $(\C^2,\omega)$ is given by
\begin{align*}
f:(\C^2,\omega)&\to \C,\\
(z_1,z_2)&\mapsto z_1z_2.
\end{align*}
This map is a submersion at every value except $0\in \C$, with regular cylindrical fibers symplectomorphic to $(T^*S^1,d\la_{S^1})$. We denote the unique singular fiber by $C_{nodal}:=f^{-1}(0)$. We refer to $f$ as a symplectic Lefschetz fibration in the sense that the fibers, $f^{-1}(c)$ for any $c\in \C$, are symplectic with respect to the same symplectic form restricted to the fiber: $\omega|_{f^{-1}(c)}$, cf.~e.g.~\cite{McDuff17_IntroSympl}.

Now take a star-shaped region $S\sse \C^2$ such that its boundary is diffeomorphic to $S^3$, $\partial S \cong S^3$, and such that the restriction of $f$ to $\partial S$ is a fibration of $\partial S$ over the unit circle $S^1\sse \C$:
\begin{align*}
F:=f|_{\partial S}:\partial S&\to S^1,\\
(z_1,z_2)&\mapsto z_1z_2.
\end{align*}
Note that $\partial S$ is in fact a contact manifold with the contact structure given by $\ker (\lambda|_{\partial S})$. A sufficient condition for having a tight contact structure is if the contact manifold is symplectically fillable, and so $\ker (\lambda|_{\partial S})$ is tight since $(S,d\la|_S)$ is an exact symplectic filling of $(\partial S,\ker (\lambda|_{\partial S}))$. Given that $S^3$ only has one tight contact structure \cite{Eliashberg89_OvertwistedStructures}, we have a contactomorphism
$$(\partial S,\ker (\lambda|_{\partial S}))\cong (S^3,\xi_{st}:=TS^3\cap iTS^3).$$

Then for any fiber $F^{-1}(c)$, $c\in S^1$, we have a symplectomorphism
$$(F^{-1}(c),d\lambda|_{F^{-1}(c)})\cong (T^*S^1,d\lambda_{S^1})$$
with coordinates $(p,\theta)\in T^*S^1$ and standard Liouville form $\lambda_{S^1}:= pd\theta$. Now consider a closed exact Lagrangian submanifold $K\sse (T^*S^1,d\lambda_{S_1})$ where $\lambda_{S^1}|_K = dg$ for some smooth function $g:K\to \R$. This implies $g$ is periodic, and so we can equivalently say $\la_{S^1}|_K$ is exact for some function $G:K\to S^1$. In the contactization of a fiber
$(S^1\times T^*S^1,\xi_{T^*S^1})$, which embeds into $(\partial S,\ker(\la|_{\partial S}))$,
with coordinates $(\psi,p,\theta)\in S^1\times T^*S^1$ and $\xi_{T^*S^1}:=ker(d\psi -\lambda_{S^1})$, define $\Lambda_{K,G} :=\{\big(G(p,\theta),p,\theta\big):(p,\theta)\in K\}$. $\Lambda_{K,G}$ is in fact a Legendrian submanifold as 
$$(d\psi - \lambda_{S^1})|_{\Lambda_K} = d\psi|_{\Lambda_K}-\lambda_{S^1}|_{\Lambda_K} = dG-dG = 0.$$

As such, we define two Legendrian unknots
$$\Lambda_1:= (e^{i\pi/4},0,\theta)\in S^1\times T^*S^1,\quad \Lambda_2:=(e^{i7\pi/4},0,\theta)\in S^1\times T^*S^1$$
and fix our Legendrian Hopf link to be modeled by $\Lambda_h := \Lambda_1\cup \Lambda_2$. Note that using the identifications above, $\La_h$ is a Legendrian submanifold of $(S^3,\xi_{st})$ and is a fibered link with respect to the map $F$ (and in turn $f$ as well), intersecting the fibers $f^{-1}(e^{i\pi/4})$ and $f^{-1}(e^{i7\pi/4})$.

\begin{center}
	\begin{figure}[h!]
		\centering
		\includegraphics[scale=.8]{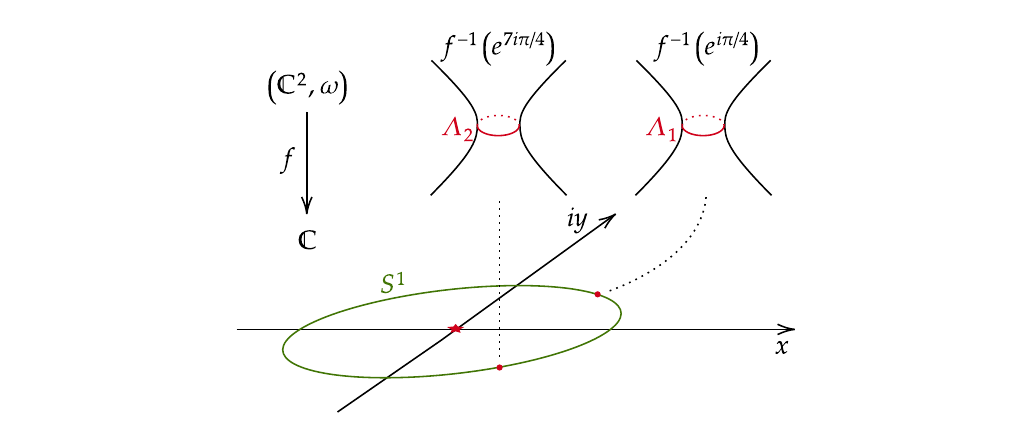}
		\caption{A standard Legendrian Hopf link $\La_h=\La_1\cup\La_2$, in red, each component $\La_i$ which is the zero section of a regular fiber for the standard Lefschetz fibration $f:\C^2\lr\C$.\vspace{.15cm}}
		\label{fig:fibered_Lambda}
	\end{figure}
\end{center}

Now through the symplectization of $(\partial S, \ker(\lambda|_{\partial S})\cong (S^3,\xi_{st})$, we argue that this star-shaped region $S$ is exact symplectomorphic to the standard Darboux ball:
$$(S,d\lambda|_S)\cong (B^4_1,\omega_0=d\lambda_0),$$
with coordinates $(x_1,y_1,x_2,y_2)\in B^4_1$, $\omega_0 = dx_1\wedge dy_1+dx_2\wedge dy_2$, and $\lambda_0=x_1dy_1+x_2dy_2$. Throughout the paper consider all $k$-dimensional balls of radius $r$, $B^k_r\sse \C^n$, to be centered at the origin in $\C^n$. For our setting, we want to think of our Lagrangian fillings as living in the standard Darboux ball $(B^4_1,\omega_0)$ where $\partial \overline{B^4_1} = S^3$. It will be useful for us to consider the ``outside'' of $S^3$ as well, and so we can think of $(B^4_1,\omega_0)$ as living ambiently in a larger 4-ball of some arbitrarily large radius, say $r=10$, $(B^4_{10},\omega_{0})$. To that end, we have the following symplectic identification between $(B^4_{10}, \omega_0)$ and $(\C P^2\sm \ell_\infty,\omega_{FS})$:
\begin{align*}
\varphi:(\C P^2\sm \ell_\infty,\omega_{FS})&\overset{\cong}{\to} (B^4_{10}, \omega_0),\\
[z_1:z_2:1]&\mapsto \frac{10}{\sqrt{1+||z_1||^2+||z_2||^2}}(z_1,z_2),
\end{align*}
with inverse
\begin{align*}
\varphi^{-1}:(B^4_{10}, \omega_0)&\overset{\cong}{\to} (\C P^2\sm \ell_\infty,\omega_{FS}),\\
(\wt{z_1},\wt{z_2})&\mapsto \frac{1}{\sqrt{100-||\wt{z_1}||^2-||\wt{z_2}||^2}}[\wt{z_1}:\wt{z_2}:1].
\end{align*}
Now the Lefschetz fibration then takes the form
$$(f\circ\varphi^{-1}):(B^4_{10}, \omega_0)\to\C,$$
$$(\wt{z_1},\wt{z_2})\mapsto \frac{z_1z_2}{100-||\wt{z_1}||^2-||\wt{z_2}||^2}.$$

A key aspect of this paper is to establish appropriate fibrations of $(B^4_1,\omega_0)$ compatible with Lagrangian fillings. Much of the theory to establish these fibrations requires us to work in this ambient space of $(\C P^2\sm \ell_\infty,\omega_{FS})$ and then we restrict to the domain $(B^4_1,\omega_0)$. And so the Lefschetz fibration can be expressed as
$$\Tilde{f}:(B^4_{1}, \omega_0)\to \C,$$
$$(\wt{z_1},\wt{z_2})\mapsto \big(f\circ \varphi^{-1}|_{B^4_1}\big)(\wt{z_1},\wt{z_2})$$

\subsection{Lagrangian Fillings}\label{sec:Lagr_def} By definition, a \textit{Lagrangian submanifold} $\psi:L\to (M,\omega)$ of a symplectic manifold $(M,\omega)$ is a half-dimensional submanifold such that the symplectic form vanishes: $2\dim L = \dim M$ and $\psi^*\omega = 0$. For a choice of Liouville form $\lambda \in \Omega^1(M)$ such that $d\lambda = \omega$, we say a Lagrangian submanifold is \textit{exact} when $\psi^*\lambda$ is an exact 1-form on $L$. An exact \textit{Lagrangian filling} $L$ of a Legendrian link $\La\sse (S^3,\xi_{st}=\ker \alpha_0)$ is an exact Lagrangian submanifold $\psi:L\to (B^4,\omega_0 = d\lambda_0)$ such that the following asymptotic condition to $\La$ holds: given that the positive and negative ends of $(B^4,d\lambda)$ agree with the positive and negative ends of the symplectization $(S^3\times \R, d(e^t\alpha_0))$ where $t$ is the coordinate on $\R$, then $L$ is asymptotic to $\Lambda$ if there exists $T\in \R$, $T>0$ such that
$$Ends(L):=L\cap \big((T,\infty)\times S^3\big) = (T,\infty)\times \Lambda.$$

This paper is concerned with classifying the exact Lagrangian fillings of the Legendrian Hopf link $\Lambda$ up to a compactly supported \textit{Hamiltonian isotopy}. A Hamiltonian isotopy of $(B^4,\omega_0)$ is a symplectic isotopy $\phi^t_{H_t}:(B^4,\omega_0)\times [0,1]\to (B^4,\omega_0)$ where the smooth family of vector fields that generate the isotopy $X_t:B^4\to TB^4$ is a family of Hamiltonian vector fields. That is, the infinitesimal generator $\iota (X_t)\omega_0$ is exact and equal to $\iota (X_t)\omega_0 = -dH_t$ for a family of smooth functions $H_t:B^4\times [0,1]\to \R$ called \textit{Hamiltonians}. Requiring each Hamiltonian $H_s$, for a fixed $s$, to be compactly supported is a necessary condition to establish different isotopy classes. Indeed, there exist examples of non-compactly supported Hamiltonian isotopies of Lagrangian fillings that exchange isotopy classes, see e.g.~ \cite{Casals22_braidloopsinfinitemonodromy}.

Note that an exact Lagrangian isotopy, i.e.~a smooth path of exact Lagrangian embeddings $L_t\sse (M,\omega)$ in a symplectic manifold, extends to a Hamiltonian isotopy, see e.g.~ \cite[Theorem 3.6.7]{Oh15_SymplTopFloerHom}. In some cases throughout this paper it is more straightforward to construct exact Lagrangian isotopies than global Hamiltonian isotopies.

In the case at hand, the topology of an exact Lagrangian filling $L$ of the Legendrian Hopf link $\Lambda_h = \Lambda_1 \cup \Lambda_2$ can be related to the \textit{Thurston-Bennequin number} ($tb$) of the components $\Lambda_1,\Lambda_2$. For a Legendrian knot in a contact $3$-manifold, $K\sse (M,\xi)$, $tb(K)$ is the twisting of the contact framing relative to the surface framing of $K$ \cite[Section 3.5]{Geiges08_IntroContact}. In our setting we have that $tb(\Lambda_i)=-1$, $i\in\{1,2\}$, and $\La_h=\La_1\cup \La_2$ is contactomorphic to the standard Legendrian Hopf link depicted in \Cref{fig:maxtb_hopf}.

Now we describe how we can compute the smooth topology of an exact Lagrangian filling $L$ of $\La_h$ by \cite[Theorem 1.4]{Chantraine10_LagrConcordance}. We perform a surgery as in \cite{Gompf98_HandlebodyStein} to $\overline{B^4}$ by attaching 2-handles $D^2\times D^2$ to each link component with the surgery associated to the framing $tb(\Lambda_i)-1=-2$ to obtain a new Stein manifold $X$. The two Lagrangian disks $D^2\times \{0\}$ now cap off our filling $L$ to get a new closed Lagrangian $L'\sse X$. The surgery description tells us that the self-intersection number $L'$, see e.g. ~\cite[Proposition 4.5.11]{GompfStipsicz99_KirbyCalc}, is
$$L'\cdot L' = \begin{pmatrix}1&1\end{pmatrix}\begin{pmatrix}-2&1\\
1&-2\end{pmatrix}\begin{pmatrix}1\\
1\end{pmatrix} = -2.$$
Now, given that for Lagrangian submanifolds the normal bundle is isomorphic to the tangent bundle through an orientation reversing isomorphism, $NL'\cong TL'$, we have the following equality for the euler classes:
$$L'\cdot L' = e(NL')[L'] = -e(TL')[L'] = -\chi (L').$$
Therefore the Euler characteristic of $L'$ equals $2$, so $L'$ is smoothly a sphere. Hence $L$ is in fact an open cylinder in $B^4$ after removing the two capping disks. In summary, an embedded exact Lagrangian filling of the standard Legendrian Hopf link $\La_h$ must be a Lagrangian cylinder.

\subsection{Constructing Standard Lagrangian Fillings}\label{sec:std_fillings}
Using the standard Lefschetz fibration $\tilde{f}$ as in \Cref{sec:std_lef_fibr} we will construct two exact Lagrangian fillings which will act as standard representatives for the two Hamiltonian isotopy classes, though we will make no claims yet as to why they are not Hamiltonian isotopic.

First, say that a Lagrangian filling $L$ is \textit{compatible} with a fibration $g:(B^4_1,\omega_0)\to \C$ if the restriction of $g$ to $L$ is a smooth $S^1$-bundle over the simple curve $g(L)\sse \C$. Given that $\La$ fibers over the map $f$ as discussed in \Cref{sec:std_lef_fibr}, we can construct exact Lagrangian fillings compatible with $\tilde{f} = (f\circ \varphi^{-1}|_{B^4_1})$, the Lefschetz fibration $f$ following a symplectomorphism to $(B^4_{10},\omega_0)$ and a restriction to $(B^4_1,\omega_0)$, in the following way: first choose an embedded curve $\sigma'\sse D^2_1 \sm\{0\} \sse \C$ with endpoints $e^{i\pi/4}=f(\Lambda_1),e^{i7\pi/4}=f(\Lambda_2)\sse S^1\sse \C$ satisfying a tangency to $f\big((T,\infty)\times\Lambda\big)$ for some $T>0$; i.e. the ends of $\sigma'$ take the form $\{e^{i\pi/4}\}\times I$ and $\{e^{i\pi/4}\}\times I$ where $I$ is some line segment in $\C$. 

Following $\varphi^{-1}|_{B^4_1}$ to get back the standard symplectic Darboux ball $(B^4_1,\omega_0)$, we have $\sigma:=(\varphi^{-1}|_{B^4_1})(\sigma')$ as an embedded open curve in the base of the Lefschetz fibration $\tilde{f}:(B^4_1,\omega_0)\to \C$. Finally, let $L_{\sigma}\sse (B^4_1,\omega_0)$ be the smooth $S^1$-bundle over $\sigma$ by choosing the zero section in each of the fibers over $\sigma$. That is, we intersect $\tilde{f}^{-1}(\sigma)$ with the hypersurface $\{||\wt{z_1}||^2-||\wt{z_2}||^2=0\}\sse (B^4_1,\omega_0)$.

\begin{center}
	\begin{figure}[h!]
		\centering
		\includegraphics[scale=.8]{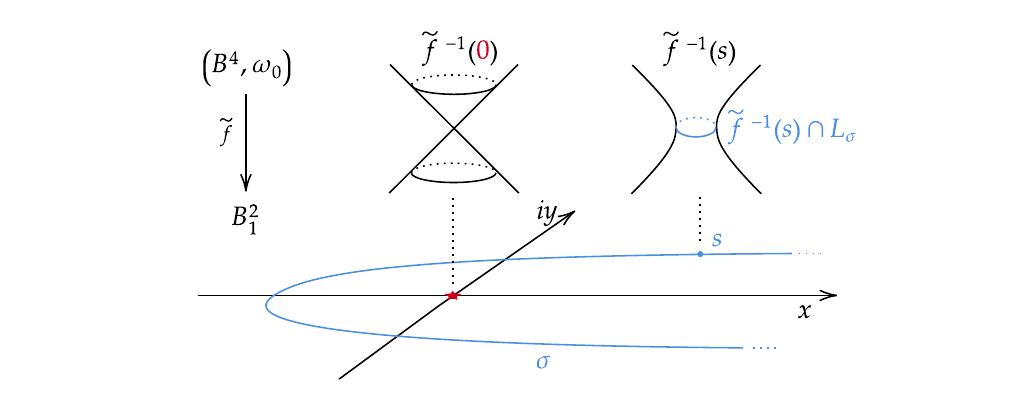}
		\caption{An exact Lagrangian filling $L_\sigma$ of $\Lambda$ compatible with the standard Lefschetz fibration $f$.\vspace{.2cm}}
		\label{fig:std_chek_fibr}
	\end{figure}
\end{center}

\begin{lemma}\label{lem:std_fillings_true}
$L_{\sigma}\sse (B^4,\omega_0=d\lambda)$ is an embedded, exact Lagrangian filling of $\La_h\sse (S^3,\xi_{st})$ for any  embedded curve $\sigma\sse B^2_1\sm \{0\}$.
\end{lemma}

%rephrease
\begin{proof}
The claim that $L_\sigma$ is a Lagrangian submanifold can be seen through the characteristic distribution of the the hypersurface $\{||\wt{z_1}||^2-||\wt{z_2}||^2=0\}$. The Euclidean gradient of the function $||\wt{z_1}||^2-||\wt{z_2}||^2$ equals $-2i\frac{d}{dt}(e^{it}\wt{z_1},e^{-it}\wt{z_2})$, which tells us that the vector field $\frac{d}{dt}(e^{it}\wt{z_1},e^{-it}\wt{z_2})$ generates the aforementioned characteristic distribution. Given that $L_\sigma$ is foliated by closed integral curves of the form $\theta \mapsto (e^{i\theta} a, e^{-i\theta}b)$ where $(a,b)\in L_\sigma$, it follows that $\omega_0|_{L_\sigma} = 0$.

Now if we fix a generator $e_0$ of $H_1(L_\sigma)$ in the class $\theta \mapsto (e^{i\theta} a, e^{-i\theta}b)$ for some fixed choice of $(a,b)\in L_\sigma$, then $||a||^2-||b||^2=0$ and we see that
$$\int_{e_0}\lambda = \pi(||a||^2-||b||^2)=0.$$
As such $L_\sigma$ is an exact Lagrangian submanifold of $(B^4,d\lambda)$.

Finally we check that the asymptotic condition to $\Lambda_h$ is satisfied. Given that we chose the base curve $\sigma$ so that its ends look like $\{e^{i\pi/4}\}\times I$ and $\{e^{i7\pi/4}\}\times I$ where $I$ is some arbitrary open interval, we see that
$$Ends(L_\sigma) = \tilde{f}^{-1}\big(\tilde{f}(\Lambda_h)\times I\big)\cap \{||\wt{z_1}||^2-||\wt{z_2}||^2=0\} = (T,\infty)\times \Lambda,$$
for some $T>0$, which completes the claim.
\end{proof}

Now the construction above defines a whole family of exact Lagrangian fillings corresponding to the family of curves satisfying this asymptotic condition to intervals near $\tilde{f}(\Lambda_h)$. For the two representatives we want to construct, we choose $\sigma_1$ to be a curve which goes around zero and $\sigma_2$ to be a curve which does not. Formally, if we let $C\sse S^1$ be an arc from $e^{i\pi/4}$ to $e^{i7\pi/4}$ through a clockwise rotation, define the \textit{extended winding number} of a curve $\sigma$ to be the winding number of the closed curve $\sigma\cup C$ around $\{0\}\in \C$. Now fix $\sigma_{Cl}$ to be some curve with extended winding number 1 and $\sigma_{Ch}$ so that it has extended winding number 0.

\begin{center}
	\begin{figure}[h!]
		\centering
		\includegraphics[scale=.8]{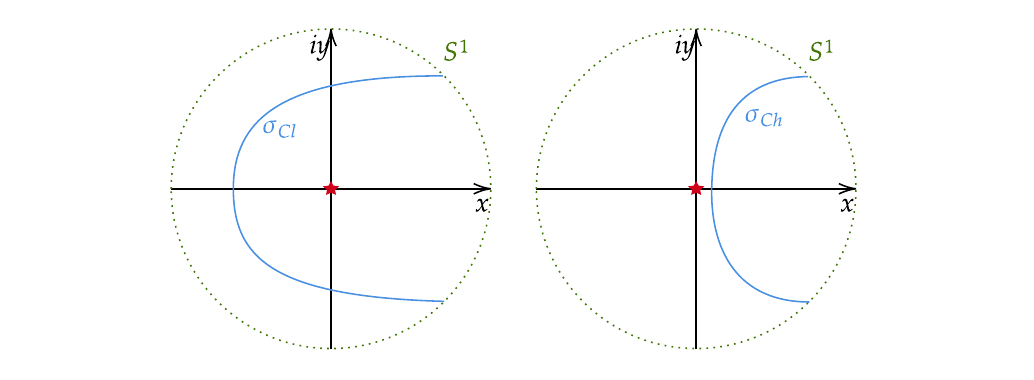}
		\caption{(Left) A choice of $\sigma_{Cl}$, a curve with extended winding number 1, and (Right) a choice of $\sigma_{Ch}$, a curve with extended winding number $0$. These yield two embedded exact Lagrangian fillings of the standard Legendrian Hopf link $\La_h$.\vspace{.15cm}}
		\label{fig:Cl_and_Ch_curves}
	\end{figure}
\end{center}

And so we have two candidates for our Hamiltonian isotopy classes of fillings: $L_{Cl}:=L_{\sigma_{Cl}}$ and $L_{Ch}:=L_{\sigma_{Ch}}$ which we will call exact Lagrangian fillings of \textit{Clifford}-type and \textit{Chekanov}-type respectively. This naming convention is in reference to the two Hamiltonian isotopy classes of Lagrangian tori in $(\C^2,\omega)$, see e.g. \cite{Chekanov96_LagrangianTori}, as well as the related notion from \cite{DR19_WhitneySphere} of Clifford-type and Chekanov-type \textit{tori} in the space $(\C P^2\sm (\ell_\infty \cup f^{-1}(1)),\omega_{FS})$.

\subsection{At Least Two Isotopy Classes of Fillings}\label{sec:sheafClassification}
As mentioned in \Cref{sec:intro}, certain algebraic viewpoints have been used establish that there are \textit{at least} two embedded, exact Lagrangian fillings of $\Lambda_h$. Using the Legendrian differential graded algebra \cite{Chekanov97_LDGA} with $\Z_2$ coefficients, \cite[Section 8]{Ekholm12_LagrangianCobord} established that there are at least a Catalan number, $C_n$, of embedded Lagrangian fillings for a $(2,n)$-torus link $\La$ through the functorial, injective map that sends isotopy classes of embedded, exact Lagrangian fillings to augmentations of $\Lambda$. This result for $(2,n)$-torus links was strengthened by \cite{Pan17_ExactLagrFillings} to the case of $\Z_2[H_1(L)]$ coefficients as well as \cite{Casals22_braidloopsinfinitemonodromy} for $\Z[H_1(L)]$ coefficients.

Similar results have been obtained from the sheaf-theoretic perspective such as in \cite{STZ16_LegKnotsandSheaves} or \cite{Casals22_LegWeaves}. For the specific example of the Legendrian Hopf link, see \cite[Examples 2.5.2 and 2.5.3]{Casals24_Microlocal} where Lagrangian fillings can be interpreted as a toric chart within a
quotient stack of sheaves of singular support with microlocal rank 1. 

In all of these examples, two distinct isotopy classes of embedded exact Lagrangian fillings are established for the Legendrian Hopf link. In this paper we aim to prove that every embedded exact Lagrangian filling of $\La_h$ is Hamiltonian isotopic to one of the standard fillings $L_{Ch}$ or $L_{Cl}$ as defined in \Cref{sec:std_fillings}, which in turn with the results mentioned above, classifies all possible embedded fillings for $\La_h$. The proof will rely on compatible fibrations, though the procedure will be the opposite of what we have done to construct these standard fillings. Instead of being given a fibration and constructing only a class fillings, we will be given an arbitrary filling and aim to construct a compatible fibration. This process uses the theory of pseudoholomorphic curves and Symplectic Field Theory; in particular, a neck stretching procedure, also called a splitting construction (see \Cref{sec:neckstretch} for the specific description). This type of argument has been used to establish obstructions to Lagrangian embeddings such as in \cite[Theorem 1.7.5]{EGH00_SFT}, \cite{Hind04_LagrSpheresS2xS2}, \cite{Hind14_Polydisks}.

\subsection{Pseudoholomorphic Conic Pencils}\label{conicpencil}
Throughout this section we will work in the compactification of $(B_{10}^4,\omega_0)$ to $(\C P^2, \omega_{FS})$ (more details for this compactification are explained in  \Cref{sec:fibration}). This is motivated by the fact that the theory of pseudoholomorphic curves is more developed and well understood in this compact setting and this will allow us to use already-established results.

Assume we are given a tame almost complex structure $J$ on $(\C P^2, \omega_{FS})$ that coincides with the standard complex structure $i$ near the line at infinity $\ell_\infty\sse \C P^2$. We call a pseudoholomorphic \textit{line} a degree one pseudoholomorphic curve inside $\C P^2$ and a pseudoholomorphic \textit{conic} a degree two curve. 

In \cite{Gromov85_PseudoholCurves} Gromov established a major result in symplectic topology in which he showed that $\C P^2$ is foliated by pseudoholomorphic lines for any choice of tame almost complex structure. The key idea behind establishing similar results for higher degree curves relies on the understanding of how curves can degenerate while deforming the almost complex structure. Gromov's compactness theorem \cite{Gromov85_PseudoholCurves} tells us that a sequence of (finite energy) pseudoholomorphic curves in a fixed homology class limits to a pseudoholomorphic curve in the same homology class but now with the possibility of having a finite set of bubbles. In our case, \cite[Lemma 4.2]{DR19_WhitneySphere}  tells us that a pseudoholomorphic conic, is either: a smoothly embedded sphere; a nodal sphere that is the union of two pseudoholomorphic lines; or a double-branched cover of a pseudoholomorphic line. Dimitroglou-Rizell discovered that when you consider pseudoholomorphic conics with a tangency restriction to mimic the curves coming from the standard Lefschetz fibration, an analogous result to Gromov's foliation by lines exists, now for conics.

First we introduce some notation: fix $q_1=[1:0:0],q_2=[0:1:0]\in \ell_\infty\sse \C P^2$ and consider the tangent vectors $v_i\sse T_{q_i}\C P^2$, $i=1,2$, to the lines
$$\ell_1 = \{[Z_1:0:Z_3]\in \C P^2\},\quad \ell_2=\{[0:Z_2:Z_3]\in \C P^2\}.$$

Now denote by $\mathcal{M}_J(v_1,v_2)$ the moduli space of $J$-holomorphic conics satisfying the two tangency conditions to $v_i$ at $q_i\in \ell_\infty$. Note that the algebraic conics coming from the standard Lefschetz fibration $f$ as described in \Cref{sec:std_lef_fibr} satisfy these tangency conditions and therefore are contained in $\mathcal{M}_{J_0}(v_1,v_2)$ where $J_0$ is the standard complex structure $J_0=i$.

Now for a smooth family of tame almost complex structures $J_t$ on $(\C P^2,\omega_{FS})$, all of which are standard near $\ell_\infty$, denote $C_{nodal}^{J_t}\sse \C P^2$ the union of the two unique $J_t$-holomorphic lines satisfying the tangencies to $v_i$. The claim that there are precisely two such lines follows from Gromov's foliation by lines of $\C P^2$ \cite{Gromov85_PseudoholCurves}. Given that $\ell_\infty$ is holomorphic under our assumption, $C_{nodal}^{J_t}$ intersects the line transversely precisely at the two points $q_1,q_2\in \ell_\infty$. Thus the nodal point $x_{nodal}^t$ of $C_{nodal}^{J_t}$ is contained inside $\C P^2\sm \ell_\infty$.

We use the following theorem in our pursuit of obtaining a symplectic fibration compatible with a given arbitrary exact Lagrangian filling of $\La_h$:

\begin{thm}[Theorem 4.3 in \cite{DR19_WhitneySphere}]\label{thm:DRfibr} The conics in $\mathcal{M}_{J_t}(v_1,v_2)$ form a smooth foliation of $\C P^2\sm (\ell_\infty \cup x_{nodal}^t)$ with symplectic leaves and a unique nodal fiber $C_{nodal}^{J_t}$. Furthermore, there is an induced family of symplectic fibrations $f_{J_t}: (\C P^2\sm \ell_\infty,\omega_{FS})\to \C$ which are submersive outside of the singularity of the nodal conic, and also depend smoothly on the parameter $t$.
\end{thm}

\section{Pseudoholomorphic Fibration Compatible with a given Lagrangian filling}\label{sec:fibration}

Throughout this section let $L\sse (B^4,\omega_{0})$ be an embedded exact Lagrangian filling of the max-$tb$ Legendrian Hopf link $\La_h\sse(S^3,\xi_{st})$. Our aim in this section is construct a pseudoholomorphic conic fibration that is compatible with such an arbitrary filling $L$, as defined in  \Cref{sec:std_fillings}. We establish the following result:

\begin{thm}\label{lem:compmain}
Let $L\sse (B^4,\omega_{0})$ be an embedded exact Lagrangian filling of the max-$tb$ Legendrian Hopf link $\La_h\sse(S^3,\xi_{st})$. Then, there exists a tame almost complex structure $J$ on $(B^4,\omega_{0})$ that gives rise to a $J$-holomorphic conic fibration $f_J:(B^4,\omega_{0})\to \C$ that is compatible with $L$. Furthermore, the embedded curve $\sigma:=f_J(L)$ is disjoint from $0=f_J(C_{nodal}^J)\in \C$ and has an extended winding number of $0$ or $1$ around $0\in \C$, and $L$ intersects any fiber over $\sigma$ along an exact Lagrangian curve in the fiber for the standard choice of Liouville form. 
\end{thm}

\begin{center}
	\begin{figure}[h!]
		\centering
		\includegraphics[scale=.8]{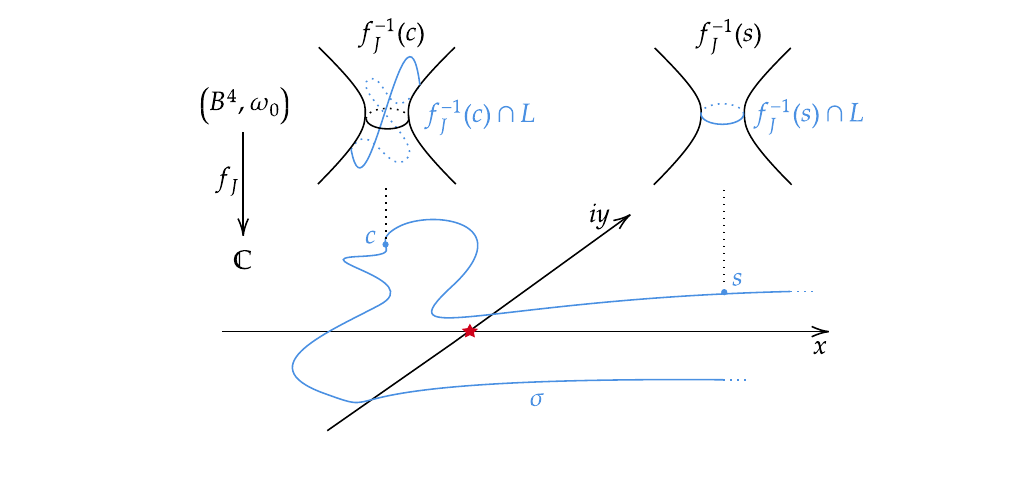}
		\caption{An exact Lagrangian filling $L$ of $\Lambda_h$ fibered over a curve $\sigma$ in the base of a pseudoholomorphic conic fibration. For any $s\in \text{Ends}(\sigma)$, $L$ intersects the fiber $f_J^{-1}(s)$ precisely along the zero section.\vspace{.55cm}}
		\label{fig:compat_fibr}
	\end{figure}
\end{center}

The goal of this section is to prove \Cref{lem:compmain}. To start, due to the asymptotic condition of any exact Lagrangian filling $L$ to $\La_h\sse (S^3,\xi_{st})$, $L$ has partial compatibility with the standard Lefschetz fibration. Formally we have:

\begin{lemma}\label{lem:compends} The ends of $L$, $Ends(L)$, are compatible with the standard Lefschetz fibration $\tilde{f}:(B_1^4,\omega_0)\to \C$ where $\sigma_{Ends}:=\tilde{f}(Ends (L))$ are two disjoint properly embedded open curves, one near $\tilde{f}(\Lambda_1)=e^{i\pi/4}\in \C$ and the other near $\tilde{f}(\Lambda_2)=e^{i7\pi/4}\in \C$. Furthermore, $L$ intersects any of the fibers $f^{-1}(s)$ for $s\in \sigma_{\text{Ends}}$ precisely along the zero section.
\end{lemma}

\begin{proof}
  The first statement is a direct consequence of the asymptotic requirement of a Lagrangian filling. We know $Ends(L)=(T,\infty)\times \Lambda_h$ for some $T>0$, so its image under the map becomes
  $$Ends(\sigma)=\tilde{f}(Ends(L)) = \tilde{f}((T,\infty)\times \Lambda_h) = (\tilde{f}(\Lambda_1)\times I_1)\cup (\tilde{f}(\Lambda_2)\times I_2),$$
  for some open line segments $I_1,I_2\sse \C$.

  The second statement follows from the fact that $\Lambda$ was chosen to be a fibered link with respect to this fibration, intersecting the fibers along the zero section.
\end{proof}

\begin{remark}\label{rmk:defineM}
Following this, let us define $M:=\overline{B^4_{1-\delta}}\sse B^4_1$ where $\delta>0$ is chosen so that $L\sm Ends(L)$ is contained inside $M$. Our goal is to deform the fibration within this region so that we obtain compatability, while maintaining our necessary asymptotic condition to $\Lambda_h$. 
\end{remark}

\subsection{Preliminaries}\label{sec:prelim} A \textit{neck stretching procedure} (also called a splitting construction) is a useful tool in Symplectic Field Theory which can be used to study the topology of Lagrangian embeddings in symplectic manifolds, see e.g.~\cite{MR3157146,EGH00_SFT}. For the case at hand, we want to use it to deform our conic fibration in order to obtain a compatible pseudoholomorphic conic fibration for any given arbitrary embedded, exact Lagrangian filling $L$ of $\La_h$. The key idea is that the cosphere bundle of a Lagrangian submanifold $P\sse (\C P^2,\omega_{FS})$ is a contact hypersurface that splits $(\C P^2,\omega_{FS})$ into a symplectic manifold diffeomorphic to $(\C P^2\sm P,\omega_{FS})\sqcup (T^*P,d\lambda_{P})$. By stretching the neck, we obtain limits of pseudoholomorphic curves being a \textit{broken pseudoholomorphic curve}, which is a collection of punctured pseudoholomorphic spheres where each punctured sphere appears in one of the connected components of the split manifold, which we call $\textit{levels}$. The punctures correspond to Reeb orbits of the cosphere bundle which, in turn, are in bijection with oriented geodesics of $P$. Hence through the compactification of $\C P^2\sm P$ by adding back in $P$, the components of a broken pseudoholomorphic curve in the top level $\C P^2\sm P$ are completed into (un-broken) pseudoholomorphic curves in $(\C P^2,\omega_{FS})$ which intersect $P$ along the closed geodesics corresponding to the punctures.

In order to make use of some of the results from \cite{DR19_WhitneySphere}, we will complete our filling, a Lagrangian cylinder (see \Cref{sec:Lagr_def}), into a Lagrangian torus and work in the compactification $\C P^2$. As such, many of our results in this section will be similar to those in \cite[Section 5]{DR19_WhitneySphere}, though it is important to note that because of the difference in setting, many of our statements will require new proofs.

\begin{remark}
Instead of using a compactification, another approach would be to consider a pseudoholomorphic fibration of $(B^4,\omega_0)$ similar to \Cref{thm:DRfibr} and run the neck stretching procedure for $L\sse (B^4,\omega_0)$ relative to the boundary. However, this would require a treatment of the asymptotics of the fibers, whereas with the approach of compactifying to $\C P^2$, the fixed points of $q_1,q_2\in \ell_\infty$ for the conics in $\mathcal{M}_J(v_1,v_2)$ (as in \Cref{thm:DRfibr}) and positivity of intersections \cite{McDuff91_Pos_of_int}, yield a bit of a more straightforward approach.
\end{remark}

Given that our filling $L$ is standard near its ends, we can choose to complete $L$ into a Lagrangian torus $\overline{L}\sse (\C P^2\sm \ell_{\infty}, \omega_{FS})$ such that the compatibility with the standard Lefschetz fibration as described in \Cref{lem:compends} extends to this added closure. Let $\sigma_T$ be a simple closed curve in $\C\sm D^2_1$ with endpoints $e^{i\pi/4},e^{i7\pi/4}\in \C$ such that $\sigma_{\text{Ends}}\cup \sigma_T$ is a smooth curve in $\C$. Then since the Lefschetz fibration $f\circ\varphi^{-1}$ is defined on all of $(B^4_{10},\omega_0)$, let $L_T\sse B^4_{10}$ be the smooth $S^1$-bundle over $\sigma_T$ by choosing the zero section in each fiber $f^{-1}(s)$ for all $s\in \sigma_T$. Again, that is defining 
$$L_T:= (f\circ\varphi^{-1})^{-1}(\sigma_T)\cap \{||\wt{z_1}||^2-||\wt{z_2}||^2=0\}.$$
Then call $\overline{L} = L\cup L_T$ the completion into a Lagrangian torus in $(B^4_{10},\omega_0)$. The Lagrangian condition follows from the same argument as in \Cref{lem:std_fillings_true}. Then we can compactify our ambient space into $(\C P^2,\omega_{FS})$ by using the symplectic identification $\varphi^{-1}$ as described in \Cref{sec:std_lef_fibr}, and adding a line at infinity $\ell_\infty$. As such, $\varphi(\overline{L})$ (which by abuse of notation we will call $\overline{L}$) is now a Lagrangian torus in $(\C P^2,\omega_{FS})$ disjoint from $\ell_\infty$.

\begin{center}
	\begin{figure}[h!]
		\centering
		\includegraphics[scale=.8]{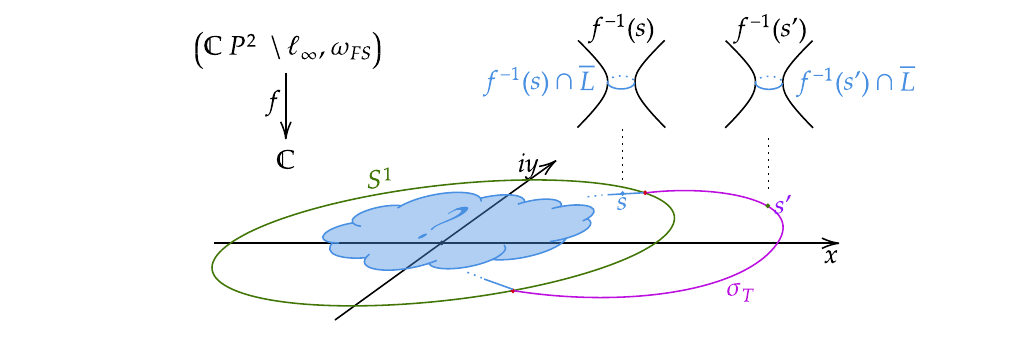}
		\caption{A completion of an embedded, exact Lagrangian filling $L$ into a Lagrangian torus $\overline{L}$ such that the component added is compatible with the standard Lefchetz fibration.\vspace{.2cm}}
		\label{fig:L_bar_fibr}
	\end{figure}
\end{center}

\begin{remark}
This torus $\overline{L}$ is not an exact Lagrangian submanifold for our standard choice of Liouville form $\lambda$. A classic result of Gromov \cite{Gromov85_PseudoholCurves} implies that there can be no closed exact Lagrangians in $(\C^2,\omega_0)$. That being said, for the generator $e_0\in H_1(\overline{L})$ induced by $H_1(L)\to H_1(\overline{L})$, the exactness assumption of our Lagrangian filling $L$ is preserved in the sense that there can be no embedded symplectic disks in $(\C P^2\sm \ell_\infty,\omega_{FS})$ with boundary in the homology class $e_0$.
\end{remark}

\begin{lemma}\label{lem:exactgen}
    For any choice of basis $\langle \wt{e_0}\rangle =H_1(L)\cong \Z$, the induced generator $e_0\in H_1(\overline{L})\cong \Z^2$ coming from $\wt{e_0}$ under the inclusion-induced map $H_1(L)\to H_1(\overline{L})$ satisfies that any symplectic disk $(D,\partial D)\to (\C P^2\sm \ell_{\infty}, \overline{L})$ with boundary in the homology class $e_0$ has symplectic area equal to $0$. Hence, there are no embedded symplectic disks in $\C P^2\sm \ell_{\infty}$ that bound $\overline{L}$ along a curve in the homology class of $e_0$.
\end{lemma}

\begin{proof}
    Since $L$ is an exact Lagrangian submanifold in $(B^4_{10},\omega_0)$ for the Liouville form $\lambda$, $\lambda|_L$ is an exact 1-form on $L$; say $\lambda|_L = dg$ for some smooth function $g:L\to \R$. Now using the symplectomorphism $\varphi:(\C P^2\sm \ell_\infty,\omega_{FS})\overset{\cong}{\to} (B^4_{10}, \omega_0)$ as described in \Cref{sec:std_lef_fibr} as well as two applications of Stokes' Theorem \cite{Spivak79_DiffGeo}, we can compute the symplectic area of any such disk $D$ directly:
    $$\int_D \omega_{FS} = \int_{\varphi(D)} (\varphi^{-1})^*\omega_{FS} = \int_{\varphi(D)} \omega_0 = \int_{e_0} \lambda =\int_{\wt{e_0}} \lambda|_L = \int_\emptyset g = 0.$$ 
\end{proof}

Now we discuss the geometric setup behind the neck stretching procedure. First fix coordinates $(\theta_1,\theta_2,p_1,p_2) = (\boldsymbol{\theta},\mathbf{p})\in T^* \mathbb{T}^2 = (S^1)^2\times \R^2$ where $\boldsymbol{\theta}$ represent the angle coordinates on $(S^1)^2$ and $\mathbf{p}$ represent the standard coordinates on $\R^2$. The standard Liouville one-form $\lambda_{\mathbb{T}^2}\in \Omega^1(T^*\mathbb{T}^2)$ can then be expressed as $\lambda_{\mathbb{T}^2}=p_1 d\theta_1+p_2 d\theta_2,$
with $d\lambda_{\mathbb{T}^2}$ the canonical symplectic form on $T^* \mathbb{T}^2$. Also for the standard flat metric on $\mathbb{T}^2$ inherited from the Euclidean metric on its covering space $\R^2$, for any $r>0$ we can consider the open co-disk bundle and its corresponding co-sphere bundle
$$T^*_r \mathbb{T}^2 = \{||\mathbf{p}||<r\} \sse T^* \mathbb{T}^2,$$
$$S^*_r \mathbb{T}^2 = \partial \overline{T^*_r \mathbb{T}^2} =  \{||\mathbf{p}||=r\}.$$
Note that the hypersurface $S^*_r \mathbb{T}^2$ is in fact a contact manifold with the pullback $\alpha_0 :=\lambda_{\mathbb{T}^2}|_{T(S^*_r \mathbb{T}^2)}$ being a contact form on the hypersurface. Identifying $S^* \mathbb{T}^2 = (S^1)^2\times S^1 \sse (S^1)^2 \times \R^2$ with angular coordinate $\psi$ for the third $S^1$ factor allows us to write the contact form as $\alpha_0=\cos(\psi)d\theta_1+\sin(\psi)d\theta_2$. The Reeb vector field $R$ on $S^* \mathbb{T}^2$ associated to this choice of contact form is uniquely determined by the equations $i_R\alpha_0 = 1$ and $i_R d\alpha_0 = 0$. Thus, the Reeb vector fields in this case are given by $R=\cos{\psi}\partial_{\theta_1}+\sin{\psi}\partial_{\theta_2}$.

A standard fact \cite{Geiges08_IntroContact} tells us that the canonical projection $S^* \mathbb{T}^2\to \mathbb{T}^2$ induces a one-to-one correspondence between Reeb orbits of R and closed oriented geodesics on $\mathbb{T}^2$ with respect to the standard flat metric on the torus. From the vector field equation above we see that these periodic Reeb orbits form 1-dimensional manifolds corresponding to $\psi$ for which $\tan(\psi)\in \Q\cup \{\infty\}$. Furthermore, given that the torus is foliated by oriented geodesics, we have that these manifolds are in bijection with the non-zero homology classes $\eta\in H^1(\mathbb{T}^2; \Q)\sm \{0\}$. And so we can denote $\Gamma_\eta \cong S^1$ the family of periodic Reeb orbits that project to oriented geodesics in the homology class $\eta$.

Weinstein's Lagrangian Neighborhood Theorem \cite{Weinstein71_NbhdTheorem} allows us to extend any closed Lagrangian embedding to a symplectic embedding of a neighborhood of the zero section of its cotangent bundle. In our case, we get a symplectic embedding
$$\phi: (T^*_{4\epsilon}\mathbb{T}^2, d\lambda_{\mathbb{T}^2}) \to (\C P^2\sm \ell_\infty, \omega_{FS}),$$
$$\phi(0_{\mathbb{T}^2}) = \overline{L},$$
which gives rise to an embedding of the co-sphere bundle $\phi(S^*_{2\epsilon} \mathbb{T}^2)$. This contact hypersurface divides our symplectic manifold $(\C P^2,\omega_{FS})$ into two components so that it is diffeomorphic to the split symplectic manifold
$$(\C P^2 \sm \overline{L}, \omega_{FS})\sqcup (T^* \mathbb{T}^2, d\lambda_{\mathbb{T}^2}).$$

\subsection{Stretching the Neck}\label{sec:neckstretch} 

Now we proceed to describe the neck-stretching procedure around an embedding of the co-sphere bundle of our Lagrangian torus $\overline{L}\sse (\C P^2\sm \ell_{\infty},\omega_{FS})$. Consider a sequence $J_\tau$, $\tau\geq 0$, of tame almost complex structures on $(\C P^2,\omega_{FS})$ satisfying:
\begin{itemize}
    \item[(i)] $J_\tau=i$ outside of the sufficiently large compact subset $M$ as defined in \Cref{rmk:defineM};
    
    \item[(ii)] in a fixed Weinstein neighborhood of $\overline{L}$
    $$\phi:(T^*_{4\epsilon}\mathbb{T}^2,d\lambda_{\mathbb{T}^2})\to (\C P^2\sm \ell_{\infty},\omega_{FS}),$$
    $$\phi(0_{\mathbb{T}^2}) = \overline{L},$$
    the almost complex structure is defined by
    $$J_\tau \partial_{\theta_i} = -\rho_\tau(||\textbf{p}||)\partial_{p_i}$$
    for a function $\rho_\tau:\R_{\geq 0}\to \R_{> 0}$ that satisfies $\rho_\tau(t) = \epsilon$ for $t\leq \epsilon$, $\rho_\tau(t) = t$ for $t\geq 3\epsilon$, and $\int_{\epsilon}^{3\epsilon}\rho(t)dt \geq \tau$; and
    
    \item[(iii)] $J_\tau$ is fixed outside of the Weinstein neighborhood $\phi(T^*_{4\epsilon}\mathbb{T}^2) \subset \C P^2$.
\end{itemize}

From these choices we specify a few important compatible almost complex structures defined on the components of the split symplectic manifold:
    \begin{enumerate}
    \item $J_{\text{std}}$ on $T^*\mathbb{T}^2$ given by $J_{\text{std}} \partial_{\theta_i} = -\rho_1(||\textbf{p}||)\partial_{p_i}$;
    
    \item $J_{\text{cyl}}$, the cylindrical almost complex structure, on the symplectization of the cosphere bundle $S^*\mathbb{T}^2$:
    $$(T^*\mathbb{T}^2 \sm 0_{\mathbb{T}^2}, d\lambda_{\mathbb{T}^2}) \cong (\R \times S^*\mathbb{T}^2, d(e^t\lambda_{\mathbb{T}^2}|_{S^*\mathbb{T}^2}))$$
    given by $J_{\text{cyl}} \partial_{\theta_i} = -||\textbf{p}||\partial_{p_i}$. This is cylindrical with respect to the contact form $\alpha := \lambda_{\mathbb{T}^2}|_{S^*\mathbb{T}^2}$; and
    
    \item $J_\infty$ on $\C P^2 \sm \overline{L}$ which coincides with $J_\text{cyl}$ inside the Weinstein neighborhood $\phi(T^*_{4\epsilon}\mathbb{T}^2)$ and with $J_\tau$ for all $\tau$ in the complement $\C P^2 \sm \phi(T^*_{4\epsilon}\mathbb{T}^2)$.
    \end{enumerate}

Stretching the neck is the procedure of taking the limit of the sequence $\{ J_\tau \}$. What essentially happens is the co-sphere bundle $S^*\overline{L}$ stretches out through its symplectization by spinning along the Reeb orbits faster and faster. A parametrization of a $J_\tau$-holomorphic curve is then being deformed in the sense that is will spend more and more time along these Reeb orbits, meaning that as we limit $\tau\to \infty$, the parametrization breaks along these orbits. 

In short, the SFT Compactness Theorem \cite{Bourgeois03_SFT} as well as \cite[Lemma 5.2]{DR19_WhitneySphere} guarantees that the limit of a sequence of $J_\tau$-holomorphic conics (resp. lines) is a broken or unbroken pseudoholomorphic conic (resp. line) which has: a finite number of embedded $J_\infty$-holomorphic punctured spheres in the \textit{Top level} $(\C P^2\sm \overline{L},\omega_{FS})$; a finite number (or zero in the case of an un-broken curve) of embedded $J_{cyl}$-holomorphic punctured spheres in the \textit{Middle level} $(\R \times S^*\mathbb{T}^2, d(e^t\lambda_{\mathbb{T}^2}|_{S^*\mathbb{T}^2}))$; and a finite number (or zero in the case of an un-broken curve) of embedded $J_{std}$-holomorphic punctured spheres in the \textit{Bottom level} $(T^*\mathbb{T}^2,d\lambda_{\mathbb{T}^2})$. In the case of a limit of lines, it is moreover the case that two different components have disjoint interiors.

An important detail to consider in this situation is the Fredholm Index of the punctured spheres. The Atiyah-Singer index theorem \cite{Atiyah-Singer63_IndexThm} tells us that the dimension of the moduli spaces containing these punctured spheres is equal to its Fredholm Index under appropriate regularity conditions.  When we choose a generic almost complex structure $J_\infty$ as we have done above, we guarantee that all $J_\infty$-holomorphic punctured spheres in $\C P^2\sm \overline{L}$ are of non-negative index \cite[Lemma 5.3]{DR19_WhitneySphere}, which allows us to establish the following result:

\begin{lemma}[Lemma 5.4 in \cite{DR19_WhitneySphere}]\label{DR19_lem5.4}
     Assume that there exist no punctured pseudoholomorphic spheres in $\C P^2 \sm \overline{L}$ of negative index. After perturbing $J_\infty$ inside some compact subset of $U\sm \ell_\infty$, where $U\sse \C P^2$ is an arbitrarily small neighborhood of $\ell_\infty$, then any curve that is either
    \begin{enumerate}
    \item a broken line satisfying a fixed point constraint at $q_i\in \ell_\infty$
    \item a broken conic satisfying the tangency conditions to $v_i$ at $q_i\in \ell_\infty$, $i\in\{1,2\}$
    \end{enumerate}
    has a top level consisting of precisely two planes of index one and possibly several cylinders of index 0. Moreover, the components passing through $\ell_\infty$ are transversely cut out when considered with the appropriate fixed point and tangency conditions, respectively.
\end{lemma}

\begin{center}
	\begin{figure}[h!]
		\centering
		\includegraphics[scale=.8]{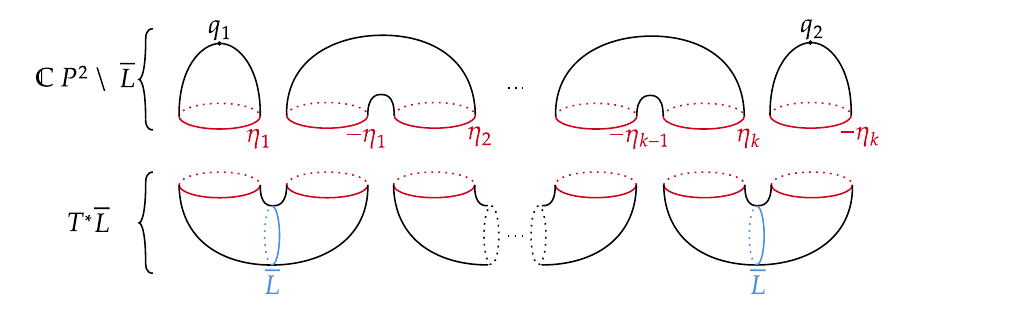}
		\caption{A generic broken pseudoholomorphic conic satisfying tangency conditions to $v_i$ at $q_i\in \ell_\infty$, $i\in\{1,2\}$. A generic broken pseudoholomorphic line looks the same except that in its top level, only one plane has a fixed point at either $q_1$ or $q_2$ while the other plane must be disjoint from $\ell_\infty$.\vspace{.2cm}}
		\label{fig:general_split_conic}
	\end{figure}
\end{center}

Our aim in this subsection is to strengthen this result in the case of broken conics by arguing that no such cylinders can exist in the top level. Then we can establish our compatible fibration by putting the two top level planes back together into a pseudoholomorphic conic along a geodesic of $\overline{L}$. This procedure is explained in detail in $\Cref{sec:proof_of_fibr}$.

\begin{center}
	\begin{figure}[h!]
		\centering
		\includegraphics[scale=.8]{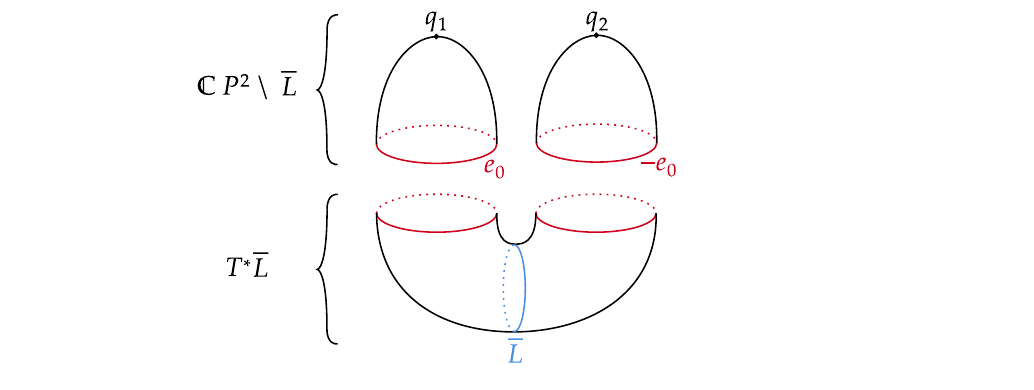}
		\caption{A generic broken pseudoholomorphic conic satisfying tangency conditions to $v_i$ at $q_i\in \ell_\infty$, $i\in\{1,2\}$ following the results of \Cref{lem:twotopplanes}}\vspace{.1cm}.
		\label{fig:resulting_conic}
	\end{figure}
\end{center}

Two key elements to establishing this strengthened result are the exactness assumption of $L$ and in turn \Cref{lem:exactgen} as well as positivity of intersection \cite{McDuff91_Pos_of_int}. Recall that in a four-dimensional symplectic manifold, two $J$-holomorphic curves $u,u'$ have only a finite number of intersection points, all of which contribute positively to the algebraic intersection number $u\bullet u'$. In particular, we will make use of this powerful result by considering the intersections of components appearing in a broken pseudoholomorphic curve with a fixed, standard conic to obtain obstructions as to what types of components these can be.

Let $C$ be a fixed conic in the region $\C P^2 \sm (\overline{L}\cup M)$ (where $M$ is defined as in \Cref{rmk:defineM}) that is a fiber over a point in $\C$ contained inside $\sigma_{T}$. By our definition of $J_\tau$, $C$ is fixed under the neck-stretching procedure and so it is an unbroken $J_\infty$-holomorphic conic given that it is also disjoint from $\overline{L}$.

\begin{lemma}[Lemma 5.8 in \cite{DR19_WhitneySphere}]\label{lem:correctgens}
 Let $\overline{L}\sse (\C P^2 \sm (\ell_\infty \cup C),\omega_{FS})$ be a Lagrangian torus such that the class $[\overline{L}]\in H_2(\C P^2 \sm (\ell_\infty \cup C))\cong \Z$ is nonzero in homology. For any choice of basis $\langle e_0, e_1 \rangle = H_1(\overline{L})$ such that $e_0$ generates the kernel of the canonical inclusion-induced map $H_1(\overline{L})\to H_1(\C P^2\sm(\ell_\infty \cup C))\cong \Z$, we have that any symplectic disk $(D,\partial D)\to (\C P^2\sm\ell_\infty, \overline{L})$ with boundary in the homology class $e_1$ satisfies $D\bullet C \neq 0$.
\end{lemma}

In order to use this result, we note that we have constructed $\overline{L}$ and chosen $C$ accordingly so that $\overline{L}$ satisfies the assumptions of this lemma where we choose $e_0$ such that it is the induced generator coming from the natural inclusion of $L\hookrightarrow \overline{L}$ as in \Cref{lem:exactgen}. Choose $e_1$ so that $D\bullet C=a_1>0$. Therefore we have that all disks in $\C P^2\sm \ell_\infty$ with boundary in the homology class $me_0+le_1\in H_1(\overline{L})$, $l> 0$, intersect with $C$ positively.

Now we will use this (along side positivity of intersection and \Cref{lem:exactgen}) to prove the following lemmas, ultimately leading to our stronger version of \Cref{DR19_lem5.4}. 

\begin{lemma}\label{lem:nodaldoesntbreak}
    The image of the nodal conic, $C_\text{nodal}:=\ell_1\cup \ell_2$, under the neck-stretching sequence is not broken. $\overline{L}$ can then be disjoined from $C_{nodal}$ by a Hamiltonian isotopy.
\end{lemma}

\begin{proof}
    Suppose that the nodal conic did break. That is, at least one of $\ell_1$ or $\ell_2$ is a broken line under the neck stretching procedure. Without loss of generality, assume $\ell_1$ breaks and let $A,B$ be the two disks in the top level of its building where $A$ satisfies the fixed point $q_1\in \ell_\infty$. Also let $C_1,...,C_{k-1}$, $k\geq 0$, be the $k-1$ cylinders in the top level and $\eta_1,...\eta_{k} \in H_1(\overline{L})$ be the geodesics corresponding to the asymptotics as in \Cref{fig:general_split_conic}. 
    
    First we note that by positivity of intersection, $B$ must necessarily be disjoint from $\ell_\infty$. Indeed if that were not the case, then completing these top level components into the cycle $\big(A\cup_{\eta_1} C_1 \cup_{\eta_2} C_2\cup ... \cup C_{k-1} \cup_{\eta_{k}} B\big)$ by gluing the components along $\eta_1,...,\eta_k$ gives us a degree 1 pseudoholomorphic sphere which would have an algebraic intersection number with $\ell_\infty$ greater than 1, a contradiction.

    As such, $B$ is a disk in $\C P^2\sm \ell_\infty$ with its asymptotic corresponding to a geodesic on $\overline{L}$ in the homology class $me_0+le_1\in H_1(\overline{L})$. Now if $l> 0$, by \Cref{lem:correctgens} we have that $B\bullet C = la_1>0$. But since $A$ and $C$ both satisfy the same tangency condition at $q_1\in \ell_\infty$ we have that $A\bullet C = 2$. So once again completing the cycle yields $\big(A\cup_{\eta_1} C_1 \cup_{\eta_2} ... \cup C_{k-1} \cup_{\eta_{k}} B\big)\bullet C \geq 2+la_1>2$, contradicting the fact that a line and a conic must have an algebraic intersection number of at most 2.

    If $l<0$, then we can compactify by adding in the cycles $\eta_1,...,\eta_k$ with opposite orientations: $\big(A\cup_{-\eta_1} C_1 \cup_{-\eta_2} ... \cup C_{k-1} \cup_{-\eta_{k}} B\big)$. In particular, this leads to $B$ being a compact disk with boundary in the homology class $-(me_0+le_1)$, and thus \Cref{lem:correctgens} implies $B\bullet C=-la_1>0$. Then the same argument as above holds.

    And if $l=0$, then $B$ is a disk in $\C P^2 \sm \ell_\infty$ that bounds $\overline{L}$ along a curve in the homology class of $e_0$ which would contradict \Cref{lem:exactgen}. Thus, $\ell_1$ (and equivalently $\ell_2$) cannot break under the neck stretching, so $C_\text{nodal}$ is not broken.

    What is left is to construct the Hamiltonian isotopy to disjoin $\overline{L}$ from $C_\text{nodal}$. We refer to \cite[Lemma 5.7]{DR19_WhitneySphere} where such an isotopy is constructed, noting that the slight difference in setting means we choose compatible almost complex structures that are standard outside of $M$, as defined in \Cref{rmk:defineM}, rather than ones that are standard near $f^{-1}(1)$.
\end{proof}

\begin{lemma}\label{lem:e_0asymptotics}
    Let $\langle e_0, e_1\rangle = H_1(\overline{L})$ be a basis as in \Cref{lem:correctgens}. Then the components of any $J_\infty$-holomorphic broken conic satisfying the tangencies at $q_i$ all have asymptotics in the homology classes $m e_0\in H_1(\overline{L})$, $m\neq 0$.
\end{lemma}

\begin{proof}
    First, let $C$ be the standard unbroken conic chosen above. The top level components of any $J_\infty$-holomorphic broken conic can be completed to a degree $2$ sphere which intersects $C$ with an algebraic intersection number of $+4$ by positivity of intersection \cite{McDuff91_Pos_of_int}. Thus, the top level components of a $J_\infty$-holomorphic broken conic satisfying the given tangencies intersect $C$ only at the two points $q_1,q_2$ on the line $\ell_\infty$ since they are non-transversal intersections, each contributing $+2$ to the algebraic intersection number.

    Now consider a top-level component $A$ of the building of the broken conic with an asymptotic corresponding to a geodesic on $\overline{L}$ in the homology class $me_0+ le_1 \in H_1(\overline{L})$. Compactify $A$ such that its boundary is in the homology class $-(me_0+le_1)\in H_1(\overline{L})$ and choose any null-homology $B\subset \C P^2 \sm \ell_\infty$ of $me_0+le_1\in H_1(\overline{L})$ to create a cycle $A\cup B$. We have produced a cycle in $H_2(\C P^2)$ with degree $d$ given by $d:=A\bullet \ell_\infty$, which can be either $0$ or $1$. Hence, the algebraic intersection number of the cycle $A\cup B$ with $C$ is $2d$. 
    
    Now since the pseudoholomorphic conics we are considering can only intersect at the two points $q_i\in \ell_\infty$, it follows that a component of one of these conics, A, can only intersect $C$ at these two points as well. Thus, $A\bullet C=2d$ using positivity of intersection. But from \Cref{lem:correctgens}, we have that $B\bullet C = la_1$ where $a_1>0$. As such, the intersection number of the cycle $A\cup B$ and $C$ is also given by $2d+la_1$. And so we get that $l$ must be $0$.
\end{proof}

Now that we know what the asymptotics look like for a broken conic, we can use our assumption of exactness on $L$ as in $\Cref{lem:exactgen}$ to classify all components appearing in the top level of a building.

\begin{lemma}\label{lem:twotopplanes}
    Any broken conic appearing as a result of stretching the neck has precisely two planes $A_i\sse \C P^2\sm \overline{L}$, $i\in\{1,2\}$ in its top level where $A_i$ is asymptotic to a geodesic on $\overline{L}$ in the homology class $\pm e_0\in H_1(\overline{L})$.
\end{lemma}

\begin{proof}
    Suppose for the sake of contradiction there was a cylinder in the top level of the building of a broken pseudoholomorphic conic. By positivity of intersection, such a cylinder is disjoint from $\ell_\infty$ and is thus contained in $\C P^2 \sm \ell_\infty$. By \Cref{lem:e_0asymptotics}, the cylinder has asymptotics corresponding to geodesics in the homology classes $k_1 e_0, k_2 e_0\in H_1(\overline{L})$, $k_1,k_2\in \Z\sm \{0\}$. Through an application of Stokes' theorem \cite{Spivak79_DiffGeo}, the symplectic area of such a cylinder is given by
    $$\int_{k_1 e_0} \lambda + \int_{k_2 e_0} \lambda.$$
    But each integral must necessarily be equal to 0 by the same argument as in \Cref{lem:exactgen}. And since all components of a broken pseudoholomorphic conic are embedded by \cite[Lemma 5.2]{DR19_WhitneySphere}, no such cylinder can exist. 
    
    Hence, the top level must consist only of the two planes $A_i\subset \C P^2\sm \overline{L}$ with tangencies  to $v_i$ at $q_i\in \ell_\infty$. And since each plane is of index 1 by \Cref{DR19_lem5.4}, we must have that their asymptotics correspond to geodesics in the homology class $\pm e_0\in H_1(\overline{L})$.
\end{proof}

\subsection{Proof of Theorem 3.1}\label{sec:proof_of_fibr}

Our proof follows very closely to the proof of \cite[Theorem 5.12]{DR19_WhitneySphere} as well as \cite[Section 5]{DR16_LagrIsoOfTori}. For completeness we recite the key details here. 

\begin{proof}
We begin by performing a Hamiltonian isotopy to disjoin $\overline{L}$ from the nodal conic $C_{nodal}$ as in \Cref{lem:nodaldoesntbreak}. We also choose $J_\infty$ generically as to satisfy \Cref{DR19_lem5.4} and proceed with our neck stretching procedure with the choices described in the previous section. SFT Compactness \cite{Bourgeois03_SFT} as well as \Cref{thm:DRfibr} give us broken pseudoholomorphic conics as limits from the moduli spaces $\mathcal{M}_{J_\tau}(v_1,v_2)$ as defined in \Cref{conicpencil}. If we fix one such broken conic, \Cref{lem:twotopplanes} tells us that this conic contains precisely two planes $A_i$, $i\in{1,2}$, in the top level where $A_i$ satisfies the tangency to $v_i$ at $q_i\in \ell_\infty$.

Let $\mathcal{M}_{J_\infty}(v_i)$ be the moduli space of $J_\infty$-holomorphic planes satisfying a tangency to $v_i$ at $q_i\in \ell_\infty$ and let $\mathcal{M}_{J_\infty}^{A_i}(v_i)$ be the component of that moduli space containing $A_i$. $\mathcal{M}_{J_\infty}(v_i)$ and $\mathcal{M}_{J_\infty}^{A_i}(v_i)$ are endowed with the topology defined in \cite[Section 8.2]{Bourgeois03_SFT} for moduli spaces of finite energy pseudoholomorphic curves in manifolds with cylindrical ends.

\begin{lemma}\label{lem:asympevalmap}
    The asymptotic evaluation map $\mathcal{M}_{J_\infty}^{A_i}(v_i)\to \Gamma \cong S^1$ which sends a plane to its asymptotic orbit is a diffeomorphism.
\end{lemma}

\begin{proof}
    First we start by showing that this moduli space is compact. If that were not the case, the SFT limit of such planes would yield at least one plane in the top level $\C P^2 \sm \overline{L}$ which, by positivity of intersection, would be disjoint from $\ell_\infty$ as well as $C$. But then by \Cref{lem:correctgens}, such a plane would bound $\overline{L}$ along an asymptotic corresponding to a geodesic in the homology class $e_0\in H_1(\overline{L})$, contradicting \Cref{lem:exactgen}.

    Now by the automatic transversality result established in \cite[Theorem 1]{Wendl09_AutoTransversality}, since $\mathcal{M}_{J_\infty}^{A_i}(v_i)$ consists of embedded planes of index 1 (calculated when considering the tangency condition to $v_i$ at $q_i\in \ell_\infty$) we have that this moduli space is transversely cut out and is submersive onto the space of orbits.

    Injectivity is finally established by the same argument as in \cite[Lemma 5.13]{DR16_LagrIsoOfTori}. By considering the asymptotic operator associated to the linearized Cauchy-Riemann operator at the puncture of one of these planes, it is shown that all positive eigenvalues of this operator have a component in the direction of the contact planes with a positive winding number. To that end, two planes that have the same asymptotic orbit but do not coincide gain an intersection number following a small holomorphic perturbation (where such a perturbation exists by the above automatic transversality result). This in turn will contradict positivity of intersection given the fact that both these planes are intersecting tangentially at $q_i\in \ell_\infty$.
\end{proof}

This tells us that for each plane in $\mathcal{M}_{J_\infty}^{A_1}(v_1)$, there exists a unique plane in $\mathcal{M}_{J_\infty}^{A_2}(v_2)$ such that both planes are asymptotic to the same Reeb orbit corresponding to the same closed unoriented geodesic on $\overline{L}$. Therefore combining the compactifications of both planes inside $\C P^2$ creates a degree two sphere intersecting $\overline{L}$ precisely along this geodesic. This sphere is not necessarily smooth at this geodesic (though it is smooth away from it), but this can be fixed by the smoothing procedure from \cite[Section 5.3]{DR16_LagrIsoOfTori}. As a result, after modifying the planes in both $\mathcal{M}_{J_\infty}^{A_i}(v_i)$, $i\in \{1,2\}$, inside the subset $U\sm \overline{L}\sse \C P^2$ where $U$ is an arbitrarily small neighborhood of $\overline{L}$ we can assume that the planes in these moduli spaces combine to form an $S^1$-family of smooth degree two spheres in $\C P^2$, each intersecting $\overline{L}$ along a unique closed curve while being tangent to $v_i$ at $q_i$. Moreover, the spheres are in fact symplectic and given that the space of tame almost complex structures on $\C P^2$ is contractible, these spheres can be assumed to be pseudoholomorphic for some global almost complex structure, call it $J$, on $(\C P^2, \omega_{FS})$, cf. ~\cite[Proposition 5.16]{DR16_LagrIsoOfTori}.

Now from \Cref{thm:DRfibr} we obtain a global symplectic fibration $F_J:(\C P^2 \sm \ell_\infty, \omega_{FS})\to \C$. Through \cite[Theorem 4.6]{DR19_WhitneySphere} (see \Cref{thm:DRnormalizedfibr} for more details), this fibration can be normalized to guarantee that the nodal conic $C_{nodal}^J$ agrees with the standard nodal conic $F^{-1}(0)$ near $\ell_\infty$ and is furthermore given by the preimage $C_{nodal}^J=F_J^{-1}(0)$. Then the symplectomorphism $\varphi$ as defined in \Cref{sec:std_lef_fibr} and restriction to $(B^4_1, \omega_0)$ gives us a symplectic fibration $f_{J}:(B^4_1, \omega_0)\to \C$ compatible with our original exact Lagrangian filling $L$. Given that we began by disjoining $\overline{L}$ from the nodal conic, we conclude that the simple curve $\sigma:=f_J(L)\sse \C$ is disjoint from $\{0\}\in \C$. The claim about extended winding number then follows from the fact that these curves are embedded.

Finally, we claim that $L$ intersects any fiber $f_J^{-1}(s)$ for $s\in \sigma$ along an exact Lagrangian closed curve of the fiber. Given that the fibration remains standard outside of the compact set $M$ as defined in \Cref{rmk:defineM}, the results of \Cref{lem:compends} still hold for the fibration $f_J$ as well. Then given that exactness is truly a homological property and $L$ intersects the fibers along curves in the same homology class $\wt{e_0}\in H_1(L)$, the intersection $L\cap f^{-1}(s)$ must be an exact Lagrangian of the fiber for all $s\in \sigma$.
\end{proof}

\begin{remark}\label{rmk:notonlyzerosection} The final paragraph of this proof does \textit{not} say that $L$ intersects each fiber $f_J^{-1}(s)$, $s\in \sigma$, precisely along the zero section of the fiber. Only near the ends of $L$ can we guarantee such a result. Fixing this will require some additional work that we tackle in \Cref{sec:vertiso}.
\end{remark}

\section{Hamiltonian Isotopies for Fibered Lagrangians}\label{sec:hamiso}
Following the previous section, through this compatible pseudoholomorphic fibration, we have now obtained a manageable description of our Lagrangian filling $L$ of $\La_h$. What is left to establish is a Hamiltonian isotopy between $L$ and one of the standard fillings $L_{Cl}$ or $L_{Ch}$ as defined in \Cref{sec:std_fillings}. We will do this in three parts: $(1)$ a Hamiltonian isotopy that brings the pseudoholomorphic fibration established in the previous section back to the standard Lefschetz fibration while maintaining compatability with our Lagrangian; $(2)$ a treatment of the fibers of the fibration to establish that any filling with a representative curve $\sigma$ is Hamiltonian isotopic to the standard exact Lagrangian fibered over $\sigma$, i.e. the Lagrangian that intersects every fiber over $\sigma$ precisely on the zero section; and finally $(3)$ a Hamiltonian isotopy in the base of the fibration to bring the representative curve $\sigma$ to one of the standard curves $\sigma_{Cl}, \sigma_{Ch} \sse \C\sm \{0\}$ as described in \Cref{sec:std_fillings}.

\subsection{Going Back to the Standard Lefschetz Fibration}\label{sec:backtolefsch}
What we have up to this point is a pseudoholomorphic conic fibration that is compatible with our given Lagrangian filling $L$. However, we want to argue that $L$ is Hamiltonian isotopic to either $L_{Cl}$ or $L_{Ch}$ which we have only defined with respect to the standard Lefschetz fibration $f:(C^2,\omega)\to \C$. To that end, we want to deform our pseudoholomorphic fibration back to the standard one while maintaining compatability with $L$. 

This is akin to the so-called Symplectic Isotopy Problem which asks whether any pseudoholomorphic curve (in $\C P^2$) is symplectically isotopic to a standard algebraic curve. Such a result is true for degree 2 curves \cite{Gromov85_PseudoholCurves} and in turn is true for all curves in $\mathcal{M}_J(v_1,v_2)$ for any tame almost complex structure $J$. The general idea behind the proof of such a statement is to argue that for the appropriate linearized $\overline{\partial}$-operator to the space of all tame almost complex structures on $(\C P^2, \omega_{FS})$, one can readily find a path of regular values from any arbitrary tame almost complex structure to the standard complex structure $J_0=i$.

But if we want to consider a symplectic isotopy of the entire fibration rather than just an individual fiber, we need to guarantee that there exists a path $J_t$ of almost complex structures such that every fiber in the fibration (except for the nodal conic) is regular for every $J_t$. Due to the work of \cite{DR19_WhitneySphere}, we do have such a guarantee:

\begin{lemma}[Lemma 4.4 in \cite{DR19_WhitneySphere}]\label{lem:DR19_lem4.4} Any smooth conic (i.e. a conic which is neither nodal nor a branched cover) in $\mathcal{M}_J(v_1,v_2)$ is a regular solution to this moduli problem for any arbitrary tame $J$.
\end{lemma}

This result was necessary for proving \Cref{thm:DRfibr}, and it in turn also gives us a normalization of our fibration and establishes a symplectic isotopy back to the standard Lefschetz fibration. Precisely, we have the following theorem at our disposal:

\begin{thm}[Theorem 4.6 in \cite{DR19_WhitneySphere}]\label{thm:DRnormalizedfibr}
Suppose we are given a pseudoholomorphic conic fibration $f_J:(\C P^2\sm\ell_\infty,\omega_{FS})\to \C$ as in \Cref{thm:DRfibr}, where $f_J$ is standard outside of some compact subset of $\C P^2\sm\ell_\infty$. It is possible to find a one-parameter family $f_{J_t}$, $t\in[0,1]$, of such conic fibrations, where $J_0=i$ and $f_{J_0}=f$, such that:
\begin{enumerate}
\item the fibers of $f_{J_1}$ coincide with the fibers of $f_J$ outside of some arbitrarily small neighborhood of the two points $q_1,q_2\in \ell_\infty$. Moreover, the almost complex structure $J_1$ can be taken to coincide with $J$ outside of an even smaller subset of that same neighborhood;
\item the nodal conics $C_{nodal}^{J_t}$ coincide with the standard nodal conic $C_{nodal}=f^{-1}(0)$ near $\ell_\infty$ and are given as the preimage $C_{nodal}^{J_t}=f_{J_t}^{-1}(0)$ for all $t$;
\item $f_{J_t}=f$ holds outside of some compact subset.
\end{enumerate}
\end{thm}

This result gives us a straightforward approach to our goal, which is to establish the following lemma:

\begin{lemma}\label{lem:backtolefsch}
For any exact Lagrangian filling $L$ of $\Lambda_h$, there exists a Hamiltonian isotopy of $(B^4,\omega_{0})$ so that $L$ becomes compatible with the standard Lefschetz fibration.
\end{lemma}

\begin{proof}
For this argument, it is still beneficial to consider the compactification of $L$ into the Lagrangian torus $\overline{L}$ inside $(\C P^2\sm \ell_\infty,\omega_{FS})$ which we used in \Cref{sec:fibration}. The proof of \Cref{lem:compmain} allows us to obtain a compatible pseudoholomorphic conic fibration $f_{J_1}:((\C P^2\sm \ell_\infty,\omega_{FS})\to \C$ with respect to $\overline{L}$. Now we apply \Cref{thm:DRnormalizedfibr} to give us a family $f_{J_t}$, $t\in [0,1]$, of conic fibrations with $J_0=i$ and $f_{J_0} = f$ which induces the desired symplectic isotopy of $(\C P^2\sm \ell_\infty,\omega_{FS})$. 

And since $J_1$ is standard outside of $M\sse \C P^2\sm \ell_\infty$, we can assume every $J_t$ to be standard outside that same subset. Thus, following the symplectic identification of $(\C P^2,\omega_{FS})\cong (B^4_{10},\omega_0)$ and a restriction to $(B^4_1,\omega_0)$, we obtain a symplectic isotopy of $(B^4_1,\omega_0)$ through symplectic conic fibrations, all of which are compatible with $L$. This final statement is a consequence of the fact that symplectic isotopies preserve Lagrangians along with the fact that $L$ intersects each individual fiber along a closed curve contained in a compact subset of the fiber. Moreover, this symplectic isotopy can be upgraded to a Hamiltonian isotopy due to \cite[Proposition 0.3]{ST05_Genus2LefschetzFibr}, which in turn tells us that such an isotopy preserves exact Lagrangians.
\end{proof}

\begin{remark}
Given that $C_{nodal}^{J_t}$ remains fixed as the fiber $f_{J_t}^{-1}(0)$ for all $t\in[0,1]$, we note that while the curve $f_{J_t}(L)$ will deform through this isotopy, its extended winding number remains fixed.
\end{remark}

\subsection{The Vertical Isotopy}\label{sec:vertiso}
We have now established that after a Hamiltonian isotopy, any exact Lagrangian filling $L$ of $\Lambda$ is compatible with the standard Lefschetz fibration. However, following the discussion in \Cref{rmk:notonlyzerosection}, outside of the ends we only know that $L$ intersects the fibers along arbitrary closed curves (though still exact Lagrangian submanifolds of the fiber). Our aim now is to construct a Hamiltonian isotopy that brings $L$ fiberwise to the zero section. 

Such a result relies on the Nearby Lagrangian Conjecture for $T^*S^1$, a conjecture of Arnol'd that states that any closed, exact Lagrangian submanifold of the cotangent bundle of a closed manifold is Hamiltonian isotopic to the zero section. The statement is known to be true for only a handful of spaces, one of which is $T^* S^1$. Given that we want to use this result on a whole one parameter family of fibers, we include a straightforward proof of the result to highlight that such an isotopy relies entirely on the embedding of the closed, exact Lagrangian curve inside $T^* S^1$.

\begin{prop}\label{lem:NLC}
        (Nearby Lagrangian Conjecture for $T^*S^1$) Every closed, exact Lagrangian submanifold of $(T^*S^1,d\lambda_{S^1})$ is Hamiltonian isotopic to the zero section.
\end{prop}

\begin{proof}
    Fix coordinates $(p,\theta)\in T^*S^1\cong \R\times S^1$ with Liouville form $\lambda=pd\theta$ and canonical symplectic form $d\lambda = dpd\theta$. Note that a closed, exact Lagrangian submanifold $\varphi:S^1\to (T^*S^1,d\lambda)$ is an essential curve in the sense that no symplectic disks bound the curve as by Stokes' theorem \cite{Spivak79_DiffGeo} we have the following:
    $$\int_{\varphi(S^1)} pd\theta = \int_{S^1}\varphi^*(pd\theta) = 0.$$
    For an inessential curve, the symplectic area of a disk bounding such a curve can be thought of as the difference between the area above and below the zero section. Now let $S$ be a non-exact Lagrangian submanifold of $(T^* S^1,d\la_{S^1})$. Crucially, there exists some equator of $T^*S^1$ such that the area spanned by $S$ above and below this equator is equal. Any equator of $T^*S^1$ can be represented by some closed $1$-form $Ad\theta\in \Omega^1 (S^1)$, $A\in \R$. Thus the vertical translation by this section defines a symplectomorphism $\tau_A:(T^*S^1,d(\lambda+Ad\theta))\to (T^*S^1,d\lambda)$, $\tau_A(p,\theta)=(p+A,\theta)$ where $\tau_A(S)$ is now an exact Lagrangian. 
    
    \begin{center}
	\begin{figure}[h!]
		\centering
		\includegraphics[scale=.8]{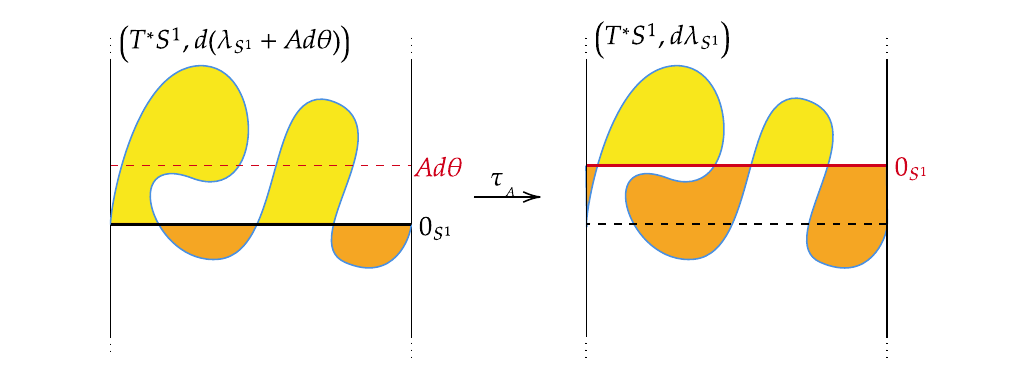}
		\caption{A non-exact Lagrangian submanifold of $(T^* S^1, d\lambda_{S^1})$ and the vertical translation by the closed 1-form $Ad\theta$ required to make it exact.\vspace{.55cm}}
		\label{fig:NLC_TS1}
	\end{figure}
\end{center}
    
    Computing for $A$, we see that $A$ depends only on the embedding of $S$:
    \begin{align*}
        0=\int_S\tau_A^*(pd\theta) = \int_{S}pd\theta&+Ad\theta=\int_S pd\theta+A\int_S d\theta,
    \end{align*}
    $$        A = -\frac{\int_S pd\theta}{\int_S d\theta}$$
    Note that if the embedding is graphical, then $A$ is precisely the negative of the symplectic area of $S$. Now let $\varphi:S^1\to (T^*S^1,d\lambda)$ be a closed, exact Lagrangian. Such an embedding is Lagrangian isotopic to the zero section; denote that isotopy by $\varphi_t:S^1\to T^*S^1$ with $\varphi_0=\varphi$ and $\varphi_1=0_{S^1}$. Now $\varphi_t$ may not be through exact Lagrangians, but the value
    $$        A_t := -\frac{\int_{\varphi_t(S^1)} pd\theta}{\int_{\varphi_t(S^1)} d\theta}$$
    varies smoothly through the isotopy with $A_0=A_1=0$ since we are starting and ending at exact Lagrangians for the standard Liouville form. As such, the composition $\tau_{A_t}\circ \varphi_t:S^1\to (T^*S^1,d\lambda)$ is now a smooth, well-defined isotopy where at every $t$, $\tau_{A_t}\circ \varphi_t(S^1)$ is an exact Lagrangian submanifold. Therefore by a standard fact, this exact Lagrangian isotopy extends to a global Hamiltonian isotopy \cite[Theorem 3.6.7]{Oh15_SymplTopFloerHom}.
\end{proof}

\begin{lemma}\label{lem:hamisotozerosection} Suppose an embedded Lagrangian filling $L$ of $\Lambda_h$ is compatible with the standard Lefschetz fibration $\Tilde{f}:(B_1^4,\omega_0)\to \C$ over the curve $\sigma=\tilde{f}(L)\sse \C\sm \{0\}$. Then there exists a Hamiltonian isotopy so that $L$ intersects each fiber over $\sigma$ precisely along the zero section.
\end{lemma}
\begin{proof}
From \Cref{lem:compends} we know that outside of $M\sse B^4_1$ (as defined in \Cref{rmk:defineM}), $L$ intersects all fibers along the zero section. Now restricting to the closed curve $\sigma_M:=\tilde{f}(L\cap M)$, and parametrizing it by $s\in[0,1]$, call $\gamma_s := L\cap \tilde{f}^{-1}(s)$ the exact curve in which $L$ intersects each fiber over $\sigma_M$. Given that the embedding $\gamma_s$ varies smoothly depending on $s$ as well as the fact that $\gamma_0,\gamma_1$ are both the zero section of their respective fibers, the vertical translation that defines the Hamiltonian isotopy in each fiber in \Cref{lem:NLC} also varies smoothly with respect to $s$. That is, the one-parameter family of fiberwise Hamiltonian isotopies is a well-defined global Hamiltonian isotopy of $(B^4_1,\omega_0)$ that brings $L$ to the zero section in each of the fibers $\tilde{f}^{-1}(s)$ for $s\in \sigma_M$.
\end{proof}

\subsection{The Horizontal Isotopy}\label{sec:horiziso}
The final step now is to argue that for any filling $L$, there exists a Hamiltonian isotopy bringing $L$ to either $L_{Cl}$ or $L_{Ch}$, our two standard representatives as established in \Cref{sec:std_fillings}. Following the previous two sections, we have that $L$ is compatible with respect to the standard Lefschetz fibration with some representative curve $\sigma\sse \C\sm \{0\}$ which has extended winding number (as defined in \Cref{sec:std_fillings}) either $0$ or $1$ around $0\in \C$ and $L$ intersects each fiber $\tilde{f}^{-1}(c)$, $c\in\sigma$ precisely along the zero section. Thus, such a Hamiltonian isotopy can be constructed by working solely in $B^2_1$ and deforming $\sigma$ to $\sigma_{Cl}$ or $\sigma_{Ch}$. A standard result tells us that for surfaces, symplectomorphisms are in bijection with area- and orientation-preserving diffeomorphisms \cite{McDuff17_IntroSympl}. This allows us to take a more visual approach to establish the required result.

\begin{lemma}\label{lem:horiziso}
Suppose an embedded Lagrangian filling $L$ of $\Lambda_h$ is compatible with the standard Lefschetz fibration $\tilde{f}:(B^4_1,\omega_0)\to \C$ over the curve $\sigma=\tilde{f}(L)\sse \C\sm \{0\}$. If $\sigma$ has extended winding number $1$ (respectively 0), then $L$ is Hamiltonian isotopic to $L_{CL}$ (respectively $L_{Ch}$).
\end{lemma}

\begin{proof}
Suppose that $\sigma$ has extended winding number 1; the proof for the Chekanov-type case is identical. We begin by applying \Cref{lem:hamisotozerosection} so that $L$ intersects each fiber $\tilde{f}^{-1}(c)$, $c\in \sigma$, along the zero section. We want to argue that $\sigma$ can be brought to $\sigma_{Cl}$ through an area-preserving diffeomorphism, where $\sigma_{Cl}$ is the representative curve used to define $L_{Cl}$ as in \Cref{sec:std_fillings}. 

Note that by the definition of a Lagrangian filling as well as \Cref{lem:compends}, $\sigma$ and $\sigma_{Cl}$ coincide in some subset of their ends. Indeed, we have that $Ends(L)=(T_1,\infty)\times \Lambda_h$ and $Ends(L_{Cl})=(T_2,\infty)\times \Lambda_h$ for some $T_1,T_2>0$. And so choosing $T=\max \{T_1,T_2\}$, we have that
$$Ends(\sigma)\cap Ends(\sigma_{Cl}) = \tilde{f}\big(Ends(L)\big)\cap\tilde{f}\big(Ends(L_{Cl})\big) = \tilde{f}\big((T,\infty)\times \Lambda_h\big).$$

For simplicity, we can redefine $Ends(\sigma)=Ends(\sigma_{Cl}) = \tilde{f}\big((T,\infty)\times \Lambda_h\big)$, which is the union of two open line segments in $\C$.

\begin{center}
	\begin{figure}[h!]
		\centering
		\includegraphics[scale=.9]{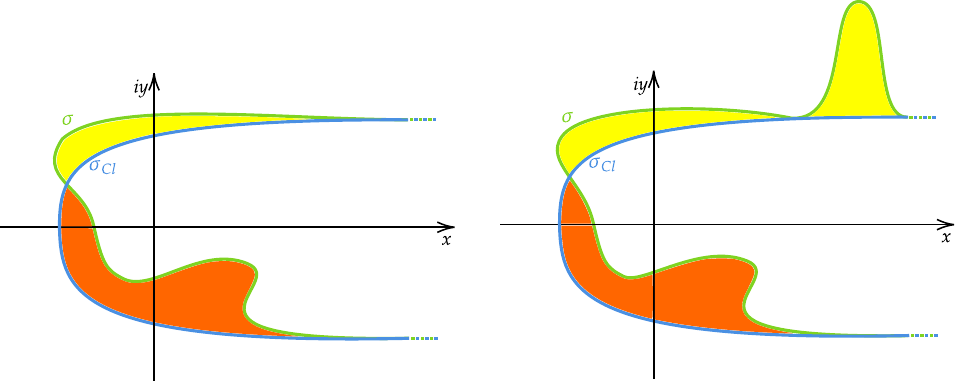}
		\caption{(Left) A curve $\sigma\sse \C$ in the base of the standard Lefschetz fibration with the yellow and orange regions being the areas needed to be displaced to bring $\sigma$ to $\sigma_{Cl}$ by an area-preserving diffeomorphism. (Right) The curve $\sigma$ following the Hamiltonian isotopy that adds a bump, which results in the yellow and orange regions having equal area.\vspace{.15cm}}
		\label{fig:horiz_iso_1}
	\end{figure}
\end{center}

Since our fillings are open submanifolds in the open symplectic manifold $(B^4_1,\omega_0)$, the following idea is possible to define a Hamiltonian isotopy: adding a bump to $\sigma$ is a Hamiltonian isotopy that adds or subtracts the necessary area in order to deform $\sigma$ to $\sigma_{Cl}$ through an area-preserving diffeomorphism. Rigorously, fix coordinates $(x,y)\in \C$ where $x+iy=z$ for any $z\in\C$ and let $\beta:\R\to \R$ be an appropriate compactly supported bump function whose graph looks like the bump added in \Cref{fig:horiz_iso_1}. Then the family of Hamiltonian functions
\begin{align*}
H_t:\C\times [0,1]&\to \R\\
(x,y,t)&\mapsto t\int_{-\infty}^x \beta(\tau)d\tau,
\end{align*}
generates a compactly supported Hamiltonian flow $\phi_t(x,y) = (x,y+t\beta(x))$ that brings the $x$-axis to the graph of $\beta(x)$. Thus, with appropriate scaling and translation applied to $H_t$, we can apply this Hamiltonian isotopy to one of the open line segments of $Ends(\sigma)$ in order to add (or subtract) the correct area so that $\sigma$ can be deformed to $\sigma_{Cl}$ through an area-preserving diffeomorphism.

Given that both $\sigma$ and $\sigma_{Cl}$ have extended winding number $0$, it follows that such a diffeomorphism extends to an isotopy $\sigma_t$ such that $\sigma_0=\sigma$, $\sigma_1=\sigma_{Cl}$, and $\sigma_t$ is disjoint from $0\in \C$ for all $t$. In particular, $\sigma_t$ has extended winding number $0$ for all $t$. Therefore, the exact Lagrangian isotopy given by 
$$L_t = \tilde{f}^{-1}(\sigma_t)\cap \{||\wt{z_1}||^2-||\wt{z_2}||^2=0\}$$ 
has $L_0=L$, $L_1=L_{Cl}$, and $L_t$ is embedded for all $t$ as we avoid the unique singular fiber $\tilde{f}^{-1}(0)$. Thus we conclude that $L$ is Hamiltonian isotopic to $L_{Cl}$ through a compactly supported isotopy that is standard outside of the subset $M\sse B^4_1$.
\end{proof}

\begin{remark}
The first Hamiltonian isotopy used to add a bump can alternatively be constructed as a contact isotopy that then extends to a Hamiltonian isotopy in the symplectization. In particular, given that the ends of our Lagrangian fillings are identically $(T,\infty)\times \Lambda$ for some $T>0$, we can pick $T<T'<\infty$ such that $\tilde{f}\big(\{T'\}\times \Lambda_2\big) = (1-\delta')e^{i7\pi/4}\in \C$ for some $\delta'\in \R$, $0<\delta'<1$. Then in the contactization of the fiber over $(1-\delta')e^{i7\pi/4}$: 
$$(\tilde{f}^{-1}((1-\delta')e^{i7\pi/4})\times S^1,\xi'),$$
we can define a contact isotopy by moving along the added copy of $S^1$ that brings $\{T'\}\times \Lambda_2\sse B^4_1$ to a new copy of $\Lambda_2$ that fibers over a different point along the circle of radius $1-\delta'$ in $\C$. Then, from \cite[Proposition 3.1]{casals17_ExoticDisks}, such a contact isotopy extends to a compactly supported Hamiltonian isotopy of the symplectization, which we note is symplectomorphic to $(B^4,\omega_0)$ after adding in the fixed point at the origin. This Hamiltonian isotopy thus deforms a neighborhood of one of the components of $Ends(L)$ in the same way as \Cref{fig:horiz_iso_1} for the appropriate choice of contact isotopy.
\end{remark}

\subsection{Conclusion}
The main result, \Cref{thm:main}, is now a direct consequence of the sequence of results from the previous sections.
\begin{proof}[Proof of \Cref{thm:main}] Let $L$ be an embedded Lagrangian filling of the Legendrian Hopf link $\Lambda_h$. Applying \Cref{lem:compmain} allows us to obtain a pseudoholomorphic conic fibration compatible with $L$ following a potential Hamiltonian isotopy of $L$ to disjoin it from the nodal conic. The Hamiltonian isotopy from \Cref{lem:backtolefsch} deforms $L$ so that it is now compatible with the standard Lefschetz fibration $\tilde{f}:(B^4_1,\omega_0)\to \C$. We then apply the vertical isotopy from \Cref{lem:hamisotozerosection} followed by the horizontal isotopy from \Cref{lem:horiziso} to establish a Hamiltonian isotopy that brings $L$ to one of our standard representatives: $L_{Cl}$ or $L_{Ch}$ as defined in \Cref{sec:std_fillings}.

Therefore, every embedded Lagrangian filling is contained in one of the (compactly supported) Hamiltonian isotopy classes $[L_{Cl}],[L_{Ch}]$. Now we claim that there are precisely two Hamiltonian isotopy classes; that is, that $L_{Cl}$ and $L_{Ch}$ are not Hamiltonian isotopic. If that were the case, then following our result, all embedded Lagrangian fillings would be Hamiltonian isotopic. But from \Cref{sec:sheafClassification}, we have that there are at least two such isotopy classes which contradicts the previous statement. This completes the classification of all embedded exact Lagrangian fillings for the max-$tb$ Legendrian Hopf link $\Lambda_h$.
\end{proof}

\color{black}
\bibliographystyle{alpha}
\bibliography{main}

@book {MR3157146,
    AUTHOR = {Abbas, Casim},
     TITLE = {An introduction to compactness results in symplectic field
              theory},
 PUBLISHER = {Springer, Heidelberg},
      YEAR = {2014},
     PAGES = {viii+252},
      ISBN = {978-3-642-31542-8; 978-3-642-31543-5},
   MRCLASS = {53D42 (53D35 57R17)},
  MRNUMBER = {3157146},
MRREVIEWER = {Janko\ Latschev},
       DOI = {10.1007/978-3-642-31543-5},
       URL = {https://doi.org/10.1007/978-3-642-31543-5},
}

@misc{casals_cbms,
      title={A microlocal introduction to Legendrian submanifolds}, 
      author={Roger Casals},
      year={2025},
      eprint={CBMS Summer School ``Legendrian Links and the Microlocal Theory of Sheaves''
},
      archivePrefix={CBMS Notes},
      primaryClass={},
      url={https://www.math.ucdavis.edu/~casals/SheavesCBMS_Notes_RCasals.pdf}, 
}

@misc{casals2023demazureweavesreducedplabic,
      title={Demazure weaves for reduced plabic graphs (with a proof that Muller-Speyer twist is Donaldson-Thomas)}, 
      author={Roger Casals and Ian Le and Melissa Sherman-Bennett and Daping Weng},
      year={2023},
      eprint={2308.06184},
      archivePrefix={arXiv},
      primaryClass={math.CO},
      url={https://arxiv.org/abs/2308.06184}, 
}

@article {Casals19_LegFronts,
    AUTHOR = {Casals, Roger and Murphy, Emmy},
     TITLE = {Legendrian fronts for affine varieties},
   JOURNAL = {Duke Math. J.},
  FJOURNAL = {Duke Mathematical Journal},
    VOLUME = {168},
      YEAR = {2019},
    NUMBER = {2},
     PAGES = {225--323},
      ISSN = {0012-7094,1547-7398},
   MRCLASS = {57R17 (53D10 53D12)},
  MRNUMBER = {3909897},
MRREVIEWER = {V.\ D.\ Sedykh},
       DOI = {10.1215/00127094-2018-0055},
       URL = {https://doi.org/10.1215/00127094-2018-0055},
}

@article {MR3784016,
    AUTHOR = {Rutherford, Dan and Sullivan, Michael G.},
     TITLE = {Generating families and augmentations for {L}egendrian
              surfaces},
   JOURNAL = {Algebr. Geom. Topol.},
  FJOURNAL = {Algebraic \& Geometric Topology},
    VOLUME = {18},
      YEAR = {2018},
    NUMBER = {3},
     PAGES = {1675--1731},
      ISSN = {1472-2747,1472-2739},
   MRCLASS = {53D42 (53D35)},
  MRNUMBER = {3784016},
MRREVIEWER = {David\ E.\ Hurtubise},
       DOI = {10.2140/agt.2018.18.1675},
       URL = {https://doi.org/10.2140/agt.2018.18.1675},
}

@article {MR3335247,
    AUTHOR = {Henry, Michael B. and Rutherford, Dan},
     TITLE = {Ruling polynomials and augmentations over finite fields},
   JOURNAL = {J. Topol.},
  FJOURNAL = {Journal of Topology},
    VOLUME = {8},
      YEAR = {2015},
    NUMBER = {1},
     PAGES = {1--37},
      ISSN = {1753-8416,1753-8424},
   MRCLASS = {57R17 (53D42 57M27)},
  MRNUMBER = {3335247},
MRREVIEWER = {Steven\ Sivek},
       DOI = {10.1112/jtopol/jtu013},
       URL = {https://doi.org/10.1112/jtopol/jtu013},
}

@article{MR4575870,
    AUTHOR = {Pan, Yu and Rutherford, Dan},
     TITLE = {Augmentations and immersed {L}agrangian fillings},
   JOURNAL = {J. Topol.},
  FJOURNAL = {Journal of Topology},
    VOLUME = {16},
      YEAR = {2023},
    NUMBER = {1},
     PAGES = {368--429},
      ISSN = {1753-8416,1753-8424},
   MRCLASS = {53D42},
  MRNUMBER = {4575870},
MRREVIEWER = {Alexander\ Fel\cprime shtyn},
       DOI = {10.1112/topo.12280},
       URL = {https://doi.org/10.1112/topo.12280},
}

@article{MR4358415,
    AUTHOR = {Casals, Roger and Gao, Honghao},
     TITLE = {Infinitely many {L}agrangian fillings},
   JOURNAL = {Ann. of Math. (2)},
  FJOURNAL = {Annals of Mathematics. Second Series},
    VOLUME = {195},
      YEAR = {2022},
    NUMBER = {1},
     PAGES = {207--249},
      ISSN = {0003-486X,1939-8980},
   MRCLASS = {57K33 (53D10)},
  MRNUMBER = {4358415},
MRREVIEWER = {Youlin\ Li},
       DOI = {10.4007/annals.2022.195.1.3},
       URL = {https://doi.org/10.4007/annals.2022.195.1.3},
}

@article{MR4194293,
    AUTHOR = {Ng, Lenhard and Rutherford, Dan and Shende, Vivek and Sivek,
              Steven and Zaslow, Eric},
     TITLE = {Augmentations are sheaves},
   JOURNAL = {Geom. Topol.},
  FJOURNAL = {Geometry \& Topology},
    VOLUME = {24},
      YEAR = {2020},
    NUMBER = {5},
     PAGES = {2149--2286},
      ISSN = {1465-3060,1364-0380},
   MRCLASS = {53D42 (53D37)},
  MRNUMBER = {4194293},
MRREVIEWER = {Alexander\ Fel\cprime shtyn},
       DOI = {10.2140/gt.2020.24.2149},
       URL = {https://doi.org/10.2140/gt.2020.24.2149},
}

@article{MR3402346,
    AUTHOR = {Bourgeois, Fr\'ed\'eric and Sabloff, Joshua M. and Traynor,
              Lisa},
     TITLE = {Lagrangian cobordisms via generating families: construction
              and geography},
   JOURNAL = {Algebr. Geom. Topol.},
  FJOURNAL = {Algebraic \& Geometric Topology},
    VOLUME = {15},
      YEAR = {2015},
    NUMBER = {4},
     PAGES = {2439--2477},
      ISSN = {1472-2747,1472-2739},
   MRCLASS = {57R17 (53D12 53D42)},
  MRNUMBER = {3402346},
MRREVIEWER = {Lenhard\ L.\ Ng},
       DOI = {10.2140/agt.2015.15.2439},
       URL = {https://doi.org/10.2140/agt.2015.15.2439},
}

@article {DR19_WhitneySphere,
    AUTHOR = {Dimitroglou Rizell, Georgios},
     TITLE = {The classification of {L}agrangians nearby the {W}hitney
              immersion},
   JOURNAL = {Geom. Topol.},
  FJOURNAL = {Geometry \& Topology},
    VOLUME = {23},
      YEAR = {2019},
    NUMBER = {7},
     PAGES = {3367--3458},
      ISSN = {1465-3060,1364-0380},
   MRCLASS = {53D12},
  MRNUMBER = {4059087},
MRREVIEWER = {Renato\ Vianna},
       DOI = {10.2140/gt.2019.23.3367},
       URL = {https://doi.org/10.2140/gt.2019.23.3367},
}

@article {McDuff91_Pos_of_int,
    AUTHOR = {McDuff, Dusa},
     TITLE = {The local behaviour of holomorphic curves in almost complex
              {$4$}-manifolds},
   JOURNAL = {J. Differential Geom.},
  FJOURNAL = {Journal of Differential Geometry},
    VOLUME = {34},
      YEAR = {1991},
    NUMBER = {1},
     PAGES = {143--164},
      ISSN = {0022-040X,1945-743X},
   MRCLASS = {53C23 (53C15 58C10)},
  MRNUMBER = {1114456},
MRREVIEWER = {Nikolai\ K.\ Smolentsev},
       URL = {http://projecteuclid.org/euclid.jdg/1214446994},
}

@article {DR16_LagrIsoOfTori,
    AUTHOR = {Dimitroglou Rizell, Georgios and Goodman, Elizabeth and Ivrii,
              Alexander},
     TITLE = {Lagrangian isotopy of tori in {$S^2\times S^2$} and
              {$\Bbb{C}P^2$}},
   JOURNAL = {Geom. Funct. Anal.},
  FJOURNAL = {Geometric and Functional Analysis},
    VOLUME = {26},
      YEAR = {2016},
    NUMBER = {5},
     PAGES = {1297--1358},
      ISSN = {1016-443X,1420-8970},
   MRCLASS = {53D12 (53D42)},
  MRNUMBER = {3568033},
MRREVIEWER = {Stefan\ Nemirovski},
       DOI = {10.1007/s00039-016-0388-1},
       URL = {https://doi.org/10.1007/s00039-016-0388-1},
}

@article {Wendl09_AutoTransversality,
    AUTHOR = {Wendl, Chris},
     TITLE = {Automatic transversality and orbifolds of punctured
              holomorphic curves in dimension four},
   JOURNAL = {Comment. Math. Helv.},
  FJOURNAL = {Commentarii Mathematici Helvetici. A Journal of the Swiss
              Mathematical Society},
    VOLUME = {85},
      YEAR = {2010},
    NUMBER = {2},
     PAGES = {347--407},
      ISSN = {0010-2571,1420-8946},
   MRCLASS = {32Q65 (53D45 57R17)},
  MRNUMBER = {2595183},
MRREVIEWER = {Umberto\ Leone\ Hryniewicz},
       DOI = {10.4171/CMH/199},
       URL = {https://doi.org/10.4171/CMH/199},
}

@article {Gromov85_PseudoholCurves,
    AUTHOR = {Gromov, M.},
     TITLE = {Pseudo holomorphic curves in symplectic manifolds},
   JOURNAL = {Invent. Math.},
  FJOURNAL = {Inventiones Mathematicae},
    VOLUME = {82},
      YEAR = {1985},
    NUMBER = {2},
     PAGES = {307--347},
      ISSN = {0020-9910,1432-1297},
   MRCLASS = {53C15 (32F25 53C57 57R15)},
  MRNUMBER = {809718},
MRREVIEWER = {Yakov\ Eliashberg},
       DOI = {10.1007/BF01388806},
       URL = {https://doi.org/10.1007/BF01388806},
}

@article {ST05_Genus2LefschetzFibr,
    AUTHOR = {Siebert, Bernd and Tian, Gang},
     TITLE = {On the holomorphicity of genus two {L}efschetz fibrations},
   JOURNAL = {Ann. of Math. (2)},
  FJOURNAL = {Annals of Mathematics. Second Series},
    VOLUME = {161},
      YEAR = {2005},
    NUMBER = {2},
     PAGES = {959--1020},
      ISSN = {0003-486X,1939-8980},
   MRCLASS = {53D35 (53D45 57R17)},
  MRNUMBER = {2153404},
MRREVIEWER = {Ivan\ Smith},
       DOI = {10.4007/annals.2005.161.959},
       URL = {https://doi.org/10.4007/annals.2005.161.959},
}

@article {Chantraine10_LagrConcordance,
    AUTHOR = {Chantraine, Baptiste},
     TITLE = {Lagrangian concordance of {L}egendrian knots},
   JOURNAL = {Algebr. Geom. Topol.},
  FJOURNAL = {Algebraic \& Geometric Topology},
    VOLUME = {10},
      YEAR = {2010},
    NUMBER = {1},
     PAGES = {63--85},
      ISSN = {1472-2747,1472-2739},
   MRCLASS = {57R17 (53D12 57M50)},
  MRNUMBER = {2580429},
MRREVIEWER = {Lenhard\ L.\ Ng},
       DOI = {10.2140/agt.2010.10.63},
       URL = {https://doi.org/10.2140/agt.2010.10.63},
}

@article {Gompf98_HandlebodyStein,
    AUTHOR = {Gompf, Robert E.},
     TITLE = {Handlebody construction of {S}tein surfaces},
   JOURNAL = {Ann. of Math. (2)},
  FJOURNAL = {Annals of Mathematics. Second Series},
    VOLUME = {148},
      YEAR = {1998},
    NUMBER = {2},
     PAGES = {619--693},
      ISSN = {0003-486X,1939-8980},
   MRCLASS = {57R17 (32E10 57R65)},
  MRNUMBER = {1668563},
MRREVIEWER = {Selman\ Akbulut},
       DOI = {10.2307/121005},
       URL = {https://doi.org/10.2307/121005},
}

@book {Oh15_SymplTopFloerHom,
    AUTHOR = {Oh, Yong-Geun},
     TITLE = {Symplectic topology and {F}loer homology. {V}ol. 1},
    SERIES = {New Mathematical Monographs},
    VOLUME = {28},
      NOTE = {Symplectic geometry and pseudoholomorphic curves},
 PUBLISHER = {Cambridge University Press, Cambridge},
      YEAR = {2015},
     PAGES = {xxiii+395},
      ISBN = {978-1-107-07245-9},
   MRCLASS = {53Dxx (32Q65 37J05)},
  MRNUMBER = {3443239},
MRREVIEWER = {Hansj\"org\ Geiges},
}

@article {Weinstein71_NbhdTheorem,
    AUTHOR = {Weinstein, Alan},
     TITLE = {Symplectic manifolds and their {L}agrangian submanifolds},
   JOURNAL = {Advances in Math.},
  FJOURNAL = {Advances in Mathematics},
    VOLUME = {6},
      YEAR = {1971},
     PAGES = {329--346},
      ISSN = {0001-8708},
   MRCLASS = {57.50},
  MRNUMBER = {286137},
MRREVIEWER = {D.\ G.\ Ebin},
       DOI = {10.1016/0001-8708(71)90020-X},
       URL = {https://doi.org/10.1016/0001-8708(71)90020-X},
}

@article {Eliashberg89_OvertwistedStructures,
    AUTHOR = {Eliashberg, Y.},
     TITLE = {Classification of overtwisted contact structures on
              {$3$}-manifolds},
   JOURNAL = {Invent. Math.},
  FJOURNAL = {Inventiones Mathematicae},
    VOLUME = {98},
      YEAR = {1989},
    NUMBER = {3},
     PAGES = {623--637},
      ISSN = {0020-9910,1432-1297},
   MRCLASS = {53C15 (53C57 58A15 58C27)},
  MRNUMBER = {1022310},
MRREVIEWER = {Yong-Geun\ Oh},
       DOI = {10.1007/BF01393840},
       URL = {https://doi.org/10.1007/BF01393840},
}

@article {casals17_ExoticDisks,
    AUTHOR = {Casals, Roger and Keating, Ailsa and Smith, Ivan},
     TITLE = {Symplectomorphisms of exotic discs},
      NOTE = {With an appendix by Sylvain Courte},
   JOURNAL = {J. \'Ec. polytech. Math.},
  FJOURNAL = {Journal de l'\'Ecole polytechnique. Math\'ematiques},
    VOLUME = {5},
      YEAR = {2018},
     PAGES = {289--316},
      ISSN = {2429-7100,2270-518X},
   MRCLASS = {57R17 (53D10)},
  MRNUMBER = {3795444},
MRREVIEWER = {Kai\ Zehmisch},
       DOI = {10.5802/jep.71},
       URL = {https://doi.org/10.5802/jep.71},
}

@article {Casals24_ClusterSeed,
    AUTHOR = {Casals, Roger and Gao, Honghao},
     TITLE = {A {L}agrangian filling for every cluster seed},
   JOURNAL = {Invent. Math.},
  FJOURNAL = {Inventiones Mathematicae},
    VOLUME = {237},
      YEAR = {2024},
    NUMBER = {2},
     PAGES = {809--868},
      ISSN = {0020-9910,1432-1297},
   MRCLASS = {53D12 (13F60 57K33)},
  MRNUMBER = {4768635},
       DOI = {10.1007/s00222-024-01268-y},
       URL = {https://doi.org/10.1007/s00222-024-01268-y},
}

@article {Casals20_LagrangianSkeleta,
    AUTHOR = {Casals, Roger},
     TITLE = {Lagrangian skeleta and plane curve singularities},
   JOURNAL = {J. Fixed Point Theory Appl.},
  FJOURNAL = {Journal of Fixed Point Theory and Applications},
    VOLUME = {24},
      YEAR = {2022},
    NUMBER = {2},
     PAGES = {Paper No. 34, 43},
      ISSN = {1661-7738,1661-7746},
   MRCLASS = {53D12 (14H20 57K33)},
  MRNUMBER = {4405603},
MRREVIEWER = {Youlin\ Li},
       DOI = {10.1007/s11784-022-00939-8},
       URL = {https://doi.org/10.1007/s11784-022-00939-8},
}

@article {Chekanov97_LDGA,
    AUTHOR = {Chekanov, Yuri},
     TITLE = {Differential algebra of {L}egendrian links},
   JOURNAL = {Invent. Math.},
  FJOURNAL = {Inventiones Mathematicae},
    VOLUME = {150},
      YEAR = {2002},
    NUMBER = {3},
     PAGES = {441--483},
      ISSN = {0020-9910,1432-1297},
   MRCLASS = {53D35 (57M27 57R17)},
  MRNUMBER = {1946550},
MRREVIEWER = {John\ B.\ Etnyre},
       DOI = {10.1007/s002220200212},
       URL = {https://doi.org/10.1007/s002220200212},
}

@article {Ekholm12_LagrangianCobord,
    AUTHOR = {Ekholm, Tobias and Honda, Ko and K\'alm\'an, Tam\'as},
     TITLE = {Legendrian knots and exact {L}agrangian cobordisms},
   JOURNAL = {J. Eur. Math. Soc. (JEMS)},
  FJOURNAL = {Journal of the European Mathematical Society (JEMS)},
    VOLUME = {18},
      YEAR = {2016},
    NUMBER = {11},
     PAGES = {2627--2689},
      ISSN = {1435-9855,1435-9863},
   MRCLASS = {53D42 (53D10 53D40 57M27 57R90)},
  MRNUMBER = {3562353},
MRREVIEWER = {Georgios\ Dimitroglou Rizell},
       DOI = {10.4171/JEMS/650},
       URL = {https://doi.org/10.4171/JEMS/650},
}

@article {Casals24_ClusterBraidVarieties,
    AUTHOR = {Casals, Roger and Gorsky, Eugene and Gorsky, Mikhail and Le,
              Ian and Shen, Linhui and Simental, Jos\'e},
     TITLE = {Cluster structures on braid varieties},
   JOURNAL = {J. Amer. Math. Soc.},
  FJOURNAL = {Journal of the American Mathematical Society},
    VOLUME = {38},
      YEAR = {2025},
    NUMBER = {2},
     PAGES = {369--479},
      ISSN = {0894-0347,1088-6834},
   MRCLASS = {13F60 (14M15 20F36)},
  MRNUMBER = {4868947},
       DOI = {10.1090/jams/1048},
       URL = {https://doi.org/10.1090/jams/1048},
}

@article {Eliashberg96_LocalLagr2Knots,
    AUTHOR = {Eliashberg, Y. and Polterovich, L.},
     TITLE = {Local {L}agrangian {$2$}-knots are trivial},
   JOURNAL = {Ann. of Math. (2)},
  FJOURNAL = {Annals of Mathematics. Second Series},
    VOLUME = {144},
      YEAR = {1996},
    NUMBER = {1},
     PAGES = {61--76},
      ISSN = {0003-486X,1939-8980},
   MRCLASS = {58F05 (57Q45 57R40)},
  MRNUMBER = {1405943},
MRREVIEWER = {Serge\ L.\ Tabachnikov},
       DOI = {10.2307/2118583},
       URL = {https://doi.org/10.2307/2118583},
}

@article {Chekanov96_LagrangianTori,
    AUTHOR = {Chekanov, Yu.\ V.},
     TITLE = {Lagrangian tori in a symplectic vector space and global
              symplectomorphisms},
   JOURNAL = {Math. Z.},
  FJOURNAL = {Mathematische Zeitschrift},
    VOLUME = {223},
      YEAR = {1996},
    NUMBER = {4},
     PAGES = {547--559},
      ISSN = {0025-5874,1432-1823},
   MRCLASS = {58D10 (58F05)},
  MRNUMBER = {1421954},
MRREVIEWER = {Karl\ Friedrich\ Siburg},
       DOI = {10.1007/PL00004278},
       URL = {https://doi.org/10.1007/PL00004278},
}

@article {Casals22_braidloopsinfinitemonodromy,
    AUTHOR = {Casals, Roger and Ng, Lenhard},
     TITLE = {Braid loops with infinite monodromy on the {L}egendrian
              contact {DGA}},
   JOURNAL = {J. Topol.},
  FJOURNAL = {Journal of Topology},
    VOLUME = {15},
      YEAR = {2022},
    NUMBER = {4},
     PAGES = {1927--2016},
      ISSN = {1753-8416,1753-8424},
   MRCLASS = {53D10 (57K10 57K33)},
  MRNUMBER = {4584583},
MRREVIEWER = {Yu\ Pan},
       DOI = {10.1112/topo.12264},
       URL = {https://doi.org/10.1112/topo.12264},
}

@article {STZ16_LegKnotsandSheaves,
    AUTHOR = {Shende, Vivek and Treumann, David and Zaslow, Eric},
     TITLE = {Legendrian knots and constructible sheaves},
   JOURNAL = {Invent. Math.},
  FJOURNAL = {Inventiones Mathematicae},
    VOLUME = {207},
      YEAR = {2017},
    NUMBER = {3},
     PAGES = {1031--1133},
      ISSN = {0020-9910,1432-1297},
   MRCLASS = {53D37 (57M27)},
  MRNUMBER = {3608288},
MRREVIEWER = {Quach thi C\^am V\^an},
       DOI = {10.1007/s00222-016-0681-5},
       URL = {https://doi.org/10.1007/s00222-016-0681-5},
}

@article {Casals22_LegWeaves,
    AUTHOR = {Casals, Roger and Zaslow, Eric},
     TITLE = {Legendrian weaves: {$N$}-graph calculus, flag moduli and
              applications},
   JOURNAL = {Geom. Topol.},
  FJOURNAL = {Geometry \& Topology},
    VOLUME = {26},
      YEAR = {2022},
    NUMBER = {8},
     PAGES = {3589--3745},
      ISSN = {1465-3060,1364-0380},
   MRCLASS = {53D35 (57R17)},
  MRNUMBER = {4562568},
MRREVIEWER = {Quach thi C\^am V\^an},
       DOI = {10.2140/gt.2022.26.3589},
       URL = {https://doi.org/10.2140/gt.2022.26.3589},
}

@misc{Casals24_Microlocal,
      title={Positive microlocal holonomies are globally regular}, 
      author={Roger Casals and Wenyuan Li},
      year={2024},
      eprint={2409.07435},
      archivePrefix={arXiv},
      primaryClass={math.SG},
      url={https://arxiv.org/abs/2409.07435}, 
}

@article {Pan17_ExactLagrFillings,
    AUTHOR = {Pan, Yu},
     TITLE = {Exact {L}agrangian fillings of {L}egendrian {$(2,n)$} torus
              links},
   JOURNAL = {Pacific J. Math.},
  FJOURNAL = {Pacific Journal of Mathematics},
    VOLUME = {289},
      YEAR = {2017},
    NUMBER = {2},
     PAGES = {417--441},
      ISSN = {0030-8730,1945-5844},
   MRCLASS = {53D42 (57R17)},
  MRNUMBER = {3667178},
MRREVIEWER = {Janko\ Latschev},
       DOI = {10.2140/pjm.2017.289.417},
       URL = {https://doi.org/10.2140/pjm.2017.289.417},
}

@book {Geiges08_IntroContact,
    AUTHOR = {Geiges, Hansj\"org},
     TITLE = {An introduction to contact topology},
    SERIES = {Cambridge Studies in Advanced Mathematics},
    VOLUME = {109},
 PUBLISHER = {Cambridge University Press, Cambridge},
      YEAR = {2008},
     PAGES = {xvi+440},
      ISBN = {978-0-521-86585-2},
   MRCLASS = {57R17 (53D35)},
  MRNUMBER = {2397738},
MRREVIEWER = {John\ B.\ Etnyre},
       DOI = {10.1017/CBO9780511611438},
       URL = {https://doi.org/10.1017/CBO9780511611438},
}

@article {Bourgeois03_SFT,
    AUTHOR = {Bourgeois, F. and Eliashberg, Y. and Hofer, H. and Wysocki, K.
              and Zehnder, E.},
     TITLE = {Compactness results in symplectic field theory},
   JOURNAL = {Geom. Topol.},
  FJOURNAL = {Geometry and Topology},
    VOLUME = {7},
      YEAR = {2003},
     PAGES = {799--888},
      ISSN = {1465-3060,1364-0380},
   MRCLASS = {53D45 (53D35 53D40 57R17)},
  MRNUMBER = {2026549},
MRREVIEWER = {Kai\ Cieliebak},
       DOI = {10.2140/gt.2003.7.799},
       URL = {https://doi.org/10.2140/gt.2003.7.799},
}

@misc{EGH00_SFT,
      title={Introduction to Symplectic Field Theory}, 
      author={Yakov Eliashberg and Alexander Givental and Helmut Hofer},
      year={2000},
      eprint={math/0010059},
      archivePrefix={arXiv},
      primaryClass={math.SG},
      url={https://arxiv.org/abs/math/0010059}, 
}

@article {Hind04_LagrSpheresS2xS2,
    AUTHOR = {Hind, R.},
     TITLE = {Lagrangian spheres in {$S^2\times S^2$}},
   JOURNAL = {Geom. Funct. Anal.},
  FJOURNAL = {Geometric and Functional Analysis},
    VOLUME = {14},
      YEAR = {2004},
    NUMBER = {2},
     PAGES = {303--318},
      ISSN = {1016-443X,1420-8970},
   MRCLASS = {53D12 (53D35)},
  MRNUMBER = {2060197},
MRREVIEWER = {Klaus\ Mohnke},
       DOI = {10.1007/s00039-004-0459-6},
       URL = {https://doi.org/10.1007/s00039-004-0459-6},
}

@article {Hind14_Polydisks,
    AUTHOR = {Hind, R. and Lisi, S.},
     TITLE = {Symplectic embeddings of polydisks},
   JOURNAL = {Selecta Math. (N.S.)},
  FJOURNAL = {Selecta Mathematica. New Series},
    VOLUME = {21},
      YEAR = {2015},
    NUMBER = {3},
     PAGES = {1099--1120},
      ISSN = {1022-1824,1420-9020},
   MRCLASS = {53D05 (53D12 53D42)},
  MRNUMBER = {3366927},
MRREVIEWER = {Yat-Sen\ Wong},
       DOI = {10.1007/s00029-013-0146-2},
       URL = {https://doi.org/10.1007/s00029-013-0146-2},
}

@book {Spivak79_DiffGeo,
    AUTHOR = {Spivak, Michael},
     TITLE = {A comprehensive introduction to differential geometry. {V}ol.
              {I}},
   EDITION = {Second},
 PUBLISHER = {Publish or Perish, Inc., Wilmington, DE},
      YEAR = {1979},
     PAGES = {xiv+668},
      ISBN = {0-914098-83-7},
   MRCLASS = {53-01},
  MRNUMBER = {532830},
MRREVIEWER = {M.\ A.\ Akivis},
}

@article {Atiyah-Singer63_IndexThm,
    AUTHOR = {Atiyah, M. F. and Singer, I. M.},
     TITLE = {The index of elliptic operators on compact manifolds},
   JOURNAL = {Bull. Amer. Math. Soc.},
  FJOURNAL = {Bulletin of the American Mathematical Society},
    VOLUME = {69},
      YEAR = {1963},
     PAGES = {422--433},
      ISSN = {0002-9904},
   MRCLASS = {57.50},
  MRNUMBER = {157392},
MRREVIEWER = {R.\ Bott},
       DOI = {10.1090/S0002-9904-1963-10957-X},
       URL = {https://doi.org/10.1090/S0002-9904-1963-10957-X},
}

@incollection{McDuff17_IntroSympl,
    author = {McDuff, Dusa and Salamon, Dietmar},
    isbn = {9780198794899},
    title = {341Area-preserving diffeomorphisms},
    booktitle = {Introduction to Symplectic Topology},
    publisher = {Oxford University Press},
    year = {2017},
    month = {03},
    doi = {10.1093/oso/9780198794899.003.0009},
    url = {https://doi.org/10.1093/oso/9780198794899.003.0009},
    eprint = {https://academic.oup.com/book/0/chapter/364194437/chapter-pdf/44350769/oso-9780198794899-chapter-9.pdf},
}

@book {GompfStipsicz99_KirbyCalc,
    AUTHOR = {Gompf, Robert E. and Stipsicz, Andr\'as I.},
     TITLE = {{$4$}-manifolds and {K}irby calculus},
    SERIES = {Graduate Studies in Mathematics},
    VOLUME = {20},
 PUBLISHER = {American Mathematical Society, Providence, RI},
      YEAR = {1999},
     PAGES = {xvi+558},
      ISBN = {0-8218-0994-6},
   MRCLASS = {57N13 (14J80 32Q55 57-02 57R17 57R57 57R65)},
  MRNUMBER = {1707327},
MRREVIEWER = {Nikolai\ N.\ Saveliev},
       DOI = {10.1090/gsm/020},
       URL = {https://doi.org/10.1090/gsm/020},
}

\end{document}